\documentclass[a4paper,11pt,english]{smfart}

\usepackage{amssymb}
\usepackage{url}
\usepackage[T1]{fontenc}
\usepackage[cal=boondoxo]{mathalfa}

\usepackage{fancybox}
\usepackage{slashbox,booktabs,amsmath}
\usepackage{makecell}

\usepackage[leftbars]{changebar}
\usepackage[textwidth=1in]{todonotes}
\usepackage{palatino}
\usepackage{rotating}
\usepackage{graphicx}

\usepackage{enumerate}
\usepackage{enumitem}

\usepackage{color}
\definecolor{shadecolor}{gray}{0.90}

\def\bfit{\bfseries\itshape}

\def\proofname{D\'emonstration}

\input xypic
\xyoption{all}
\xyoption{arc}

\makeindex

\def\indexname{Index des notations}
\newtheorem{theo}{Theorem}[section]
\def\thetheo{\thesection.\arabic{theo}}
\newtheorem{prop}[theo]{Proposition}

\newtheorem{lem}[theo]{Lemma}

\newtheorem{coro}[theo]{Corollary}

\newtheorem{defi}[theo]{Definition}

\renewcommand\theequation{\thetheo}
\def\equat{\refstepcounter{theo}\begin{equation}}
\def\endequat{\end{equation}}

\renewcommand\thetable{\thetheo}

\renewcommand\thesection{\arabic{section}}
\renewcommand\thesubsection{\thesection.\Alph{subsection}}

\def\contentsname{Table des mati\`eres}
\def\listfigurename{Liste des figures}
\def\listtablename{Liste des tables}
\def\appendixname{Appendice}

\def\refname{R\'ef\'erences}

    \def\CM{{\mathbb{C}}}

\def\HG{{\mathfrak H}}

  \def\mG{{\mathfrak m}}

\def\PG{{\mathfrak P}}  \def\pG{{\mathfrak p}}  
\def\QG{{\mathfrak Q}}    
  \def\rG{{\mathfrak r}}  
\def\SG{{\mathfrak S}}

    \def\ZM{{\mathbb{Z}}}

    \def\AC{{\mathcal{A}}}

    \def\EC{{\mathcal{E}}}
    \def\FC{{\mathcal{F}}}
\def\Gb{{\mathbf G}}    
\def\Hb{{\mathbf H}}    
    \def\IC{{\mathcal{I}}}

\def\Lb{{\mathbf L}}    \def\LC{{\mathcal{L}}}
    \def\MC{{\mathcal{M}}}

\def\Pb{{\mathbf P}}  \def\pb{{\mathbf p}}  \def\PC{{\mathcal{P}}}

\def\Sb{{\mathbf S}}    \def\SC{{\mathcal{S}}}
    
    \def\UC{{\mathcal{U}}}
    \def\VC{{\mathcal{V}}}
    
    \def\XC{{\mathcal{X}}}
    \def\YC{{\mathcal{Y}}}
\def\Zb{{\mathbf Z}}    \def\ZC{{\mathcal{Z}}}

\def\Nrm{{\mathrm{N}}}

\def\Srm{{\mathrm{S}}}    
\def\Trm{{\mathrm{T}}}

    \def\ECt{{\tilde{\mathcal{E}}}}

\def\Rhat{{\hat{R}}}

          \def\wdo{{\dot{w}}}

\def\a{\alpha}

\def\g{\gamma}

\def\d{\delta}
\def\D{\Delta}
\def\e{\varepsilon}
\def\ph{\varphi}
\def\l{\lambda}
\def\L{\Lambda}

\def\o{\omega}
\def\O{\Omega}

\def\th{\theta}

\def\t{\tau}
\def\x{\xi}
\def\z{\zeta}

          \def\phit{{\tilde{\phi}}}

\def\mub{{\boldsymbol{\mu}}}

\def\taub{{\boldsymbol{\tau}}}

\DeclareMathOperator{\Cus}{{\mathrm{Cus}}}
\DeclareMathOperator{\diag}{{\mathrm{diag}}}

\DeclareMathOperator{\Id}{{\mathrm{Id}}}

\DeclareMathOperator{\Irr}{{\mathrm{Irr}}}
\DeclareMathOperator{\Ker}{{\mathrm{Ker}}}

\DeclareMathOperator{\Res}{{\mathrm{Res}}}

\def\to{\rightarrow}
\def\longto{\longrightarrow}
\def\injto{\hookrightarrow}

\def\fonctio#1#2#3#4{\begin{array}{ccc}
{#1} & \longto & {#2} \\
{#3} & \longmapsto & {#4} 
\end{array}}

\def\vide{\varnothing}
\def\DS{\displaystyle}

\def\SSS{\scriptscriptstyle}

\def\finl{~$\blacksquare$}

\def\lexp#1#2{\kern\scriptspace\vphantom{#2}^{#1}\kern-\scriptspace#2}
\def\le{\hspace{0.1em}\mathop{\leqslant}\nolimits\hspace{0.1em}}
\def\ge{\hspace{0.1em}\mathop{\geqslant}\nolimits\hspace{0.1em}}

\mathchardef\inferieur="321E
\mathchardef\superieur="321F

\def\eqna{\begin{eqnarray*}}
\def\endeqna{\end{eqnarray*}}

\def\paramset{\nabla}

\def\itemth#1{\item[${\mathrm{(#1)}}$]}

\catcode`\@=11
\long\def\@car#1#2\@nil{#1}
\long\def\@first#1#2{#1}
\long\def\@second#1#2{#2}
\long\def\ifempty#1{\expandafter\ifx\@car#1@\@nil @\@empty
  \expandafter\@first\else\expandafter\@second\fi}
\catcode`\@=12

\def\GL{{\mathrm{GL}}}

\renewcommand{\Ref}{{\mathrm{Ref}}}

\def\boitegrise#1#2{\begin{centerline}{\fcolorbox{black}{shadecolor}{~
    \begin{minipage}[t]{#2}{\vphantom{~}#1\vphantom{$A_{\DS{A_A}}$}}
            \end{minipage}~}}\end{centerline}\medskip}

\def\ve{{\SSS{\vee}}}

\theoremstyle{remark}
\newtheorem{rema}[theo]{Remark}

\newtheorem{exemple}[theo]{Example}

\theoremstyle{plain}

\def\BIL{LR}
\def\GAUCHE{L}
\def\CAR{CAR}
\def\FAM{FAM}

\def\reg{{\mathrm{reg}}}

\def\grad{{\mathrm{gr}}}

\def\xyinj{\ar@{^{(}->}}
\def\xysur{\ar@{->>}}

\bigskip

\makeatletter
\def\hlinewd#1{%
\noalign{\ifnum0=`}\fi\hrule \@height #1 %
\futurelet\reserved@a\@xhline}
\makeatother

\newlength\epaisLigne

\def\split{{\mathrm{spl}}}

\usepackage{multirow}

\newcommand{\longiso}{\stackrel{\sim}{\longrightarrow}}
\newcommand{\longbij}{\stackrel{\sim}{\longleftrightarrow}}

\def\MCov{{\overline{\MC}}}
\def\LCov{{\overline{\LC}}}
\def\SCov{{\overline{\SC}}}
\def\ICov{{\overline{\IC}}}

\makeatletter
\def\hlinewd#1{%
\noalign{\ifnum0=`}\fi\hrule \@height #1 %
\futurelet\reserved@a\@xhline}
\makeatother

\usepackage{multirow}

\usepackage{lscape}

\newcommand{\longinjto}{\lhook\joinrel\longrightarrow}

\def\nor{{\mathrm{nor}}}
\def\la{\langle}
\def\ra{\rangle}
\def\gr{\operatorname{gr}\nolimits}

\def\gr{{\mathrm{gr}}}

\def\para{{\mathrm{Parab}}}
\def\pspec{{\mathrm{PSpec}}}
\def\pmax{{\mathrm{PMax}}}
\def\setw{{\mathrm{set}}}
\def\ptw{{\mathrm{pt}}}

\def\symp{{\mathcal{S\!y\!m\!p}}}
\def\rees{{\mathrm{Rees}}}

\usepackage{scalerel,stackengine}
\stackMath
\newcommand\reallywidehat[1]{%
\savestack{\tmpbox}{\stretchto{%
  \scaleto{%
    \scalerel*[\widthof{\ensuremath{#1}}]{\kern-.6pt\bigwedge\kern-.6pt}%
    {\rule[-\textheight/2]{1ex}{\textheight}}%WIDTH-LIMITED BIG WEDGE
  }{\textheight}% 
}{0.5ex}}%
\stackon[1pt]{#1}{\tmpbox}%
}

\begin{document}

%\baselineskip=16pt
%\large\baselineskip=20pt
%\Large\baselineskip=24pt

\title{Automorphisms and symplectic leaves\\ of Calogero-Moser spaces}

\author{{\sc C\'edric Bonnaf\'e}}
\address{IMAG, Universit\'e de Montpellier, CNRS, Montpellier, France} 

\makeatletter
\email{cedric.bonnafe@umontpellier.fr}
\makeatother

\date{\today}

\thanks{The author is partly supported by the ANR: 
Projects No ANR-16-CE40-0010-01 (GeRepMod) and ANR-18-CE40-0024-02 (CATORE)}

% 

%\tableofcontents

% \noindent{\bf Acknowledgements.} We thank warmly D. Juteau 
% for providing us the main ideas leading to Example~\ref{ex:b2-g2}.
% 
% \bigskip

\begin{abstract}
We study the symplectic leaves of the subvariety of fixed points of an automorphism of a 
Calogero-Moser space induced by an element of finite order of 
the normalizer of the associated complex reflection group. 
We give a parametrization {\it \`a la Harish-Chandra} 
of its symplectic leaves (generalizing  
earlier works of Bellamy and Losev). 
This result is inspired by the mysterious relations between 
the geometry of Calogero-Moser spaces and unipotent representations 
of finite reductive groups, which is the theme of another paper~\cite{cm-unip}.  
\end{abstract}

\maketitle

\pagestyle{myheadings}

\markboth{\sc C. Bonnaf\'e}{\sc Automorphisms and symplectic leaves}

Let $V$ be a finite dimensional vector space and let $W$ be a finite 
subgroup of $\Gb\Lb_\CM(V)$ generated by reflections. To a class 
function $k$ on $W$ supported on the set of reflections, 
Etingof and Ginzburg~\cite{EG} have associated a normal 
irreducible affine complex variety $\ZC_k(V,W)$ called 
a (generalized) {\it Calogero-Moser space}. If $\t$ is an 
element of finite order of the normalizer of $W$ in $\Gb\Lb_\CM(V)$ 
stabilizing the class function $k$, it induces an automorphism 
of $\ZC_k(V,W)$. The main theme of this paper is the study of the 
symplectic leaves of the variety $\ZC_k(V,W)^\t$ 
of its fixed points in $\ZC_k(V,W)$ (endowed with its reduced 
closed subscheme structure). 

Note that $W$ acts trivially on $\ZC_k(V,W)$ so, by replacing 
$\t$ by $w\t$ for some $w \in W$ if necessary, we may assume that 
the natural morphism $V^\t \longto (V/W)^\t$ is onto (the argument 
is due to Springer~\cite{springer} and will be recalled in~\S\ref{sub:springer}): this will 
be assumed throughout this paper and will simplify the statements.

The Poisson bracket on $\ZC_k(V,W)$ induces a 
Poisson bracket on $\ZC_k(V,W)^\t$ and we are interested in 
parametrizing the symplectic leaves of this fixed points subvariety. For this, we define a 
{\it $\t$-cuspidal} symplectic leaf (or a {\it $\t$-cuspidal} point) to be a zero-dimensional 
symplectic leaf of $\ZC_k(V,W)^\t$, 
and we define a {\it $\t$-split} 
parabolic subgroup of $W$ to be the stabilizer of some point in $V^\t$. 
We also denote by $W_\t$ the quotient $\Sigma/\Pi$, where $\Sigma$ (resp. $\Pi$) 
is the setwise (resp. pointwise) stabilizer of $V^\t$. For its action on $V^\t$, 
the group $W_\t$ is a reflection group~\cite{lehrer springer}. 
Our result is as follows:

\bigskip

\noindent{\bf Theorem A.} {\it Assume that the natural morphism $V^\t \longto (V/W)^\t$ is onto. 
Then there is a natural bijection between 
the set of symplectic leaves of $\ZC_k(V,W)^\t$ and 
the set of $W_\t$-orbits of 
pairs $(P,p)$, where $P$ is a $\t$-split parabolic subgroup 
and $p$ is a $\t$-cuspidal point of $\ZC_{k_P}(V_P,P)^\t$. 

Moreover, the dimension of the symplectic leaf associated with $(P,p)$ through this bijection 
is equal to $2 \dim (V^P)^\t$.}

\bigskip

Here, $k_P$ is the restriction of $k$ to $P$. We will give in 
Section~\ref{sec:tau-hc} an explicit description of the bijection. If $\t=1$, this 
result was proved by Bellamy~\cite{bellamy cuspidal} and Losev~\cite{losev} and 
might be viewed as a {\it Harish-Chandra theory} of symplectic 
leaves. So Theorem~A can be thought as a {\it $\t$-Harish-Chandra 
theory}, inspired by Brou\'e-Malle-Michel {\it $d$-Harish-Chandra 
theory} of unipotent representations of finite reductive groups~\cite{BMM} 
(see~\cite{cm-unip} for a further discussion of this analogy and 
applications of Theorem~A). The main point 
is to combine Springer/Lehrer-Springer theory (which describes the action 
of the setwise stabilizer of $V^\t$ on $V^\t$) with Bellamy/Losev works. 
We propose the following conjecture about the geometry of symplectic leaves 
of $\ZC_k(V,W)^\t$. 

\bigskip

\centerline{\begin{minipage}{0.7\textwidth}
\noindent{\bf Conjecture B.} {\it Let $(P,p)$ be as in Theorem~A and let $\SC$ 
denote the corresponding symplectic leaf of $\ZC_k(V,W)^\t$. 
Then there exists a parameter $l$ for the pair $((V^P)^\t,\Nrm_{W_\t}(P_\t)/P_\t)$ 
and a $\CM^\times$-equivariant isomorphism of Poisson varieties
$$\overline{\SC}^\nor \simeq \ZC_l((V^P)^\t,\Nrm_{W_\t}(P_\t)/P_\t).$$
Here, $\overline{\SC}^\nor$ denotes the normalization of the closure of $\SC$.} 
\end{minipage}}

\bigskip

Note that this conjecture is not known even in the case where $\t=\Id_V$ 
(in which case $W_\t=W$ and $P_\t=P$). 
It has been proved by Maksimau and the author~\cite{bonnafe maksimau} 
whenever $\ZC_k(V,W)$ is smooth and $\t \in W\cdot \CM^\times$. 

\bigskip

The paper is organized as follows. We recall the set-up (reflection groups, 
Cherednik algebras, Calogero-Moser spaces...) in the first section and the 
second section is a recollection of useful results on Poisson structures 
and symplectic leaves. In the third section, we recall the main results 
of Lehrer-Springer on the group $W_\t$ and some of its consequences. 
In the fourth section, we restate Theorem~A and Conjecture~B in more precise 
terms. The proof of Theorem~A is given in Sections~5 to~8 (see the end 
of Section~4 for the description of the different steps). 
We propose in the nineth section an overview of the known cases for Conjecture~B. 
A short appendix summarizes easy results about completions of rings 
that are needed in Section~8 to conclude the proof of Theorem~A.

\bigskip

\noindent{\bf Acknowledgements.} 
We thank warmly Gwyn Bellamy for sharing his files containing 
preliminary versions of his upcoming papers with Chalykh~\cite{be-ch} 
and Maksimau and Schedler~\cite{be-sc}, and for his help for understanding his work 
on symplectic leaves of Calogero-Moser spaces. 

We also wish to thank warmly 
Jean Michel for very long and fruitful discussions about $d$-Harish-Chandra 
theory of unipotent representations of finite reductive groups: even 
if it does not appear in this text (but will be explained in~\cite{cm-unip}), all what I learned 
from him was a great source of inspiration for this paper.

\bigskip

\section{Set-up}\label{sec:notation}

\medskip

\subsection{Complex numbers} 
Throughout this paper, we will abbreviate $\otimes_\CM$ as 
$\otimes$ and all varieties will be algebraic, complex, quasi-projective 
and reduced. If $\XC$ is an irreducible variety, we denote by 
$\XC^\nor$ its normalization. If $\XC$ is an affine variety, we denote by $\CM[\XC]$ 
its coordinate ring: if moreover $\XC$ is irreducible, then $\XC^\nor$ is also 
affine and $\CM[\XC^\nor]$ is the integral closure of $\CM[\XC]$ in its fraction 
field (which will be denoted by $\CM(\XC)$). 

% If $\t$ is an automorphism of the algebraic variety $\XC$, 
% we denote by $\XC^{(\t)}$ the {\it scheme} of fixed points 
% of $\t$ in $\XC$ (i.e. the fiber product of the diagonal 
% morphism $\XC \to \XC \times \XC$ with the graph of $\t$). 
% Note that it might not be reduced and we denote by $\XC^\t$ 
% its reduced subscheme: it is the closed subvariety of $\XC$ 
% whose points are the points $x \in \XC$ such that $\t(x)=x$. 
% If $\XC$ is smooth and $\t$ has finite order, then $\XC^{(\t)}$ 
% is also smooth (in particular, it is reduced and so $\XC^{(\t)}=\XC^\t$). 
% If $\XC$ is affine, let $I_\t$ denote the ideal 
% $\langle (\t(a)-a)_{a \in \CM[\XC]} \rangle$ of $\CM[\XC]$. Then 
% $\XC^{(\t)}=\Spec \CM[\XC]/I_\t$ and $\CM[\XC^\t]=\CM[\XC]/\sqrt{I_\t}$. 

\def\Nrmov{{\overline{\mathrm{N}}}}

We fix in this paper a complex vector space $V$ of finite dimension $n$. 
If $X$ is a subset of $V$ (or $V^*$), and if $G$ is a subgroup 
of $\Gb\Lb_\CM(V)$, we denote by $G_X^\setw$ (resp. $G_X^\ptw$) the setwise 
(resp. pointwise) stabilizer of $X$ and we set $G[X]=G_X^\setw/G_X^\ptw$. 
Then $G[X]$ acts faithfully on $X$ (and on the vector space spanned by $X$). 
If $X=\{v\}$ is a singleton, then $G_X^\setw=G_X^\ptw$ (and we will denote both simply 
by $G_X$ or $G_v$) and $G[X]=1$. If $H$ is a subgroup of $G$, we set 
$\Nrmov_G(H)=\Nrm_G(H)/H$. 

If moreover $G$ is finite, 
we identify $(V^G)^*$ and $(V^*)^G$, and we denote by $V_G$ 
the unique $G$-stable subspace of $V$ such that $V=V_G \oplus V^G$. 

\bigskip

\subsection{Reflections} 
Let $W$ be a finite 
subgroup of $\GL_\CM(V)$. We set
$$\Ref(W)=\{s \in W~|~\dim_\CM V^s=n-1\}$$
and note that, for the moment, we {\it do not} assume that 
$W$ is generated by $\Ref(W)$. 
We set $\e : W \to \CM^\times$, $w \mapsto \det(w)$. 
We identify $\CM[V]$ (resp. $\CM[V^*]$) with the symmetric 
algebra $\Srm(V^*)$ (resp. $\Srm(V)$).

We denote by $\AC$ the set of {\it reflecting hyperplanes} of $W$, namely
$$\AC=\{V^s~|~s \in \Ref(W)\}.$$
If $H \in \AC$, we denote by $\a_H$ an element of $V^*$ such that 
$H=\Ker(\a_H)$ and by $\a_H^\vee$ an element of $V$ such that 
$V=H \oplus \CM \a_H^\vee$ and the line $\CM\a_H^\vee$ is $W_H^\ptw$-stable.
We set $e_H=|W_H^\ptw|$. Note that $W_H^\ptw$ is cyclic of order $e_H$ and that 
$\Irr(W_H^\ptw)=\{\Res_{W_H^\ptw}^W \e^j~|~0 \le j \le e-1\}$. We denote by $\e_{H,j}$ 
the (central) primitive idempotent of $\CM W_H$ associated with the character 
$\Res_{W_H^\ptw}^W \e^{-j}$, namely
$$\e_{H,j}=\frac{1}{e_H}\sum_{w \in W_H^\ptw} \e(w)^j w \in \CM W_H^\ptw.$$
If $\O$ is a $W$-orbit of reflecting hyperplanes, we write $e_\O$ for the 
common value of all the $e_H$, where $H \in \O$. 
We denote by $\paramset$ the set of pairs $(\O,j)$ where $\O \in \AC$ and 
$0 \le j \le e_\O-1$. 
The vector space of families of complex numbers 
indexed by $\paramset$ will be denoted by $\CM^\paramset$: elements 
of $\CM^\paramset$ will be called {\it parameters}. 
If $k=(k_{\O,j})_{(\O,j) \in \paramset} \in \CM^\paramset$, we 
define $k_{H,j}$ for all $H \in \O$ and $j \in \ZM$ by 
$k_{H,j}=k_{\O,j_0}$ where $\O$ is the $W$-orbit of $H$ 
and $j_0$ is the unique element of $\{0,1,\dots,e_H-1\}$ such that 
$j \equiv j_0 \mod e_H$. 

\bigskip

\subsection{Parabolic subgroups}\label{sub:para}
We denote by $\para(W)$ the set of parabolic subgroups of $W$ 
(i.e. the set of subgroups of $W$ which are stabilizers of some point 
of $V$) and by $\para(W)/W$ the set of conjugacy classes of parabolic 
subgroups of $W$. 
If $P \in \para(W)$, we denote by 
$\VC(P)$ the set of elements 
$v \in V$ such that $W_v=P$: it is a non-empty open subset of 
$V^P$. By definition, $W_{V^P}^\ptw=P$ 
and $W_{V^P}^\setw=\Nrm_W(P)$, so that $W[V^P]=\Nrmov_W(P)$.  
The family $(\VC(P))_{P \in \para(W)}$ 
is a stratification of $V$ (the order between strata corresponds to 
the reverse order of the inclusion of parabolic subgroups). 

This stratification is stable under the action of the group $W$. 
If $\PG \in \para(W)/W$, we denote by $\UC(\PG)$ the image of 
$\VC(P)$ in $V/W$, where $P$ is any element of $\PG$. Then 
$(\UC(\PG))_{\PG \in \para(W)/W}$ is a stratification of $V/W$ (the order 
between strata corresponds to the reverse order of the inclusion, 
up to conjugacy, of parabolic subgroups). Replacing $V$ by $V^*$, 
we define similarly $\VC^*(P)$ and $\UC^*(\PG)$ for $P \in \para(W)$ 
and $\PG \in \para(W)/W$. By definition, $\Nrmov_W(P)$ acts freely on 
$\VC(P)$ or $\VC^*(P)$. Moreover, for $P \in \PG$, 
the natural map $\VC(P) \to \UC(\PG)$ induces an isomorphism of varieties 
\equat\label{eq:vp-wp-up}
\VC(P)/\Nrmov_W(P) \longiso \UC(\PG).
\endequat
In particular, $\UC(\PG)$ is smooth.

\bigskip

\subsection{Rational Cherednik algebra at ${\boldsymbol{t=0}}$}
Let $k \in \CM^\paramset$. 
We define the {\it rational Cherednik algebra $\Hb_k$} ({\it at $t=0$}) to be the quotient 
of the algebra $\Trm(V\oplus V^*)\rtimes W$ (the semi-direct product of the tensor algebra 
$\Trm(V \oplus V^*)$ with the group $W$) 
by the relations 
\equat\label{eq:rels}
\begin{cases}
[x,x']=[y,y']=0,\\
[y,x]=\DS{\sum_{H\in\mathcal{A}} \sum_{j=0}^{e_H-1}
e_H(k_{H,j}-k_{H,j+1}) 
\frac{\langle y,\a_H \rangle \cdot \langle \a_H^\ve,x\rangle}{\langle \a_H^\ve,\a_H\rangle} \e_{H,j}},
\end{cases}
\endequat
for all $x$, $x'\in V^*$, $y$, $y'\in V$. 
Here $\la\ ,\ \ra: V\times V^*\to\CM$ is the standard pairing. 
The first commutation relations imply that 
we have morphisms of algebras $\CM[V] \to \Hb_k$ and $\CM[V^*] \to \Hb_k$. 
Recall~\cite[Theo.~1.3]{EG} 
that we have an isomorphism of $\CM$-vector spaces 
\equat\label{eq:pbw}
\CM[V] \otimes \CM W \otimes \CM[V^*] \longiso \Hb_k
\endequat
induced by multiplication (this is the so-called {\it PBW-decomposition}). 

\medskip

\begin{rema}\label{rem:parametres particuliers}
Let $(l_\O)_{\O \in \AC/W}$ be a family of complex numbers and let 
$k' \in \CM^\paramset$ be defined by $k_{\O,j}'=k_{\O,j} + l_\O$. Then 
$\Hb_k=\Hb_{k'}$. This means that there is no restriction to generality 
if we consider for instance only 
parameters $k$ such that $k_{\O,0}=0$ for all $\O$, 
or only parameters $k$ such that $k_{\O,0}+k_{\O,1}+\cdots+k_{\O,e_\O-1}=0$ 
for all $\O$ (as in~\cite{calogero}).\finl
\end{rema}

\medskip

\subsection{Calogero-Moser space} 
We denote by $\Zb_k$ the center of the algebra $\Hb_k$: 
it is well-known~\cite[Theo~3.3~and~Lem.~3.5]{EG} that 
$\Zb_k$ is an integral domain, which is integrally closed. Moreover, it contains 
$\CM[V]^W$ and $\CM[V^*]^W$ as subalgebras~\cite[Prop.~3.6]{gordon} 
(so it contains $\Pb=\CM[V]^W \otimes \CM[V^*]^W$), 
and it is a free $\Pb$-module of rank $|W|$. We denote by $\ZC_k$ the 
affine algebraic variety whose ring of regular functions $\CM[\ZC_k]$ is $\Zb_k$: 
this is the {\it Calogero-Moser space} associated with the datum $(V,W,k)$. 
It is irreducible and normal. 

We set $\PC=V/W \times V^*/W$, so that $\CM[\PC]=\Pb$ and the inclusion 
$\Pb \injto \Zb_k$ induces a morphism of varieties 
$$\Upsilon_k : \ZC_k \longto \PC$$
which is finite (and flat if $W=\langle \Ref(W) \rangle$). 

% Using the PBW-decomposition, we define a $\CM$-linear map 
% $\Omeb^{\Hb_k} : \Hb_k \longto \CM W$
% by 
% $$\Omeb^{\Hb_k}(f w g)=f(0)g(0)w$$
% for all $f \in \CM[V]$, $g \in \CM[V^*]$ and $w \in \CM W$. This map is $W$-equivariant 
% for the action on both sides by conjugation, so it induces a well-defined $\CM$-linear map 
% $$\Omeb^k : \Zb_k \longto \Zrm(\CM W).$$
% Recall from~\cite[Cor.~4.2.11]{calogero} that $\Omeb^k$ is a morphism of algebras, and that 
% \equat\label{eq:lissite-omega}
% \text{\it $\ZC_k$ is smooth if and only if $\Omeb^k$ is surjective.}
% \endequat
% The ``only if'' part is essentially due to Gordon~\cite[Cor.~5.8]{gordon} 
% (but the reader must see also~\cite[Prop.~9.6.6~and~(16.1.2)]{calogero} 
% for translating Gordon's result in terms of $\Omeb^k$) while 
% the ``if'' part follows from the work of Bellamy, Schedler and Thiel~\cite[Cor.~1.4]{BST}.

\bigskip

% \subsection{Other parameters} 
% Let $\CC$ denote the space of maps $\Ref(W) \to \CM$ 
% which are constant on conjugacy classes of reflections. 
% The element
% $$\sum_{(\O,j) \in \paramset} \sum_{H \in \O} (k_{H,j}-k_{H,j+1}) e_H \e_{H,j}$$
% of $\Zrm(\CM W)$ is supported only by reflections, so there exists 
% a unique map $c_k \in \CC$ such that 
% $$\sum_{(\O,j) \in \paramset} \sum_{H \in \O} (k_{H,j}-k_{H,j+1}) e_H \e_{H,j}
% = \sum_{s \in \Ref(W)} (\e(s)-1) c_k(s) s.$$
% Then the map $\CM^\paramset \to \CC$, $k \mapsto c_k$ is linear and surjective. 
% With this notation, we have 
% \equat\label{eq:cs}
% [y,x] = \sum_{s \in \Ref(W)} (\e(s)-1)\hskip1mm c_k(s) 
% \hskip1mm
% \frac{\langle y,\a_s \rangle \cdot \langle \a_s^\ve,x\rangle}{\langle 
% \a_s^\ve,\a_s\rangle}
% \hskip1mm s,
% \endequat
% for all $y \in V$ and $x \in V^*$.
% 
% \bigskip

\subsection{Extra-structures on the Calogero-Moser space}
The Calogero-Moser space $\ZC_k$ is endowed with extra-structures 
(a $\CM^\times$-action, a Poisson bracket, a filtration, 
an action of $\Nrm_{\Gb\Lb_\CM(V)}(W)$...) 
which are described below.

\bigskip

\subsubsection{Grading, $\CM^\times$-action}\label{subsub:action}
The algebra $\Trm(V\oplus V^*)\rtimes W$ can be $\ZM$-graded in such a way that the 
generators have the following degrees
$$
\begin{cases}
\deg(y)=-1 & \text{if $y \in V$,}\\
\deg(x)=1 & \text{if $x \in V^*$,}\\
\deg(w)=0 & \text{if $w \in W$.}
\end{cases}
$$
This descends to a $\ZM$-grading on $\Hb_k$, because the defining relations~(\ref{eq:rels}) 
are homogeneous. Since the center of a graded algebra is always graded, the subalgebra $\Zb_k$ 
is also $\ZM$-graded.  So the Calogero-Moser space $\ZC_k$ 
inherits a regular $\CM^\times$-action. Note also that
by definition $\Pb=\CM[V]^W \otimes \CM[V^*]^W$ is clearly a graded 
subalgebra of $\Zb_k$. 

\bigskip

\subsubsection{Poisson structure}
Let $t \in \CM$. One can define a deformation $\Hb_{t,k}$ of $\Hb_k$ as follows: 
$\Hb_{t,k}$ is the quotient 
of the algebra $\Trm(V\oplus V^*)\rtimes W$ 
by the relations 
\equat\label{eq:rels-1}
\begin{cases}
[x,x']=[y,y']=0,\\
[y,x]=t \la y,x \ra + \DS{\sum_{H\in\mathcal{A}} \sum_{=0}^{e_H-1}
e_H(k_{H,i}-k_{H,i+1}) 
\frac{\langle y,\a_H \rangle \cdot \langle \a_H^\ve,x\rangle}{\langle \a_H^\ve,\a_H\rangle} \e_{H,i}},
\end{cases}
\endequat
for all $x$, $x'\in V^*$, $y$, $y'\in V$. It is well-known~\cite{EG} 
that the PBW decomposition still holds so 
that the family $(\Hb_{t,k})_{t \in \CM}$ is a flat deformation of $\Hb_k=\Hb_{0,k}$. 
This allows to define a Poisson bracket $\{\ ,\ \}$ on $\Zb_k$ as follows: 
if $z_1$, $z_2 \in \Zb_k$, 
we denote by $z_1^{t}$, $z_2^t$ the corresponding element of $\Hb_{t,k}$ through the 
PBW decomposition and we define 
$$\{z_1,z_2\} = \lim_{t \to 0} \frac{[z_1^t,z_2^t]}{t}.$$
Finally, note that 
\equat\label{eq:poisson equivariant}
\text{\it The Poisson bracket is $\CM^\times$-equivariant.}
\endequat
% Recall the following important result of Bellamy~\cite[Theo.~1.4]{bellamy counting}, 
% which follows works of Ginzburg-Kaledin~\cite{GK} and Namikawa~\cite{namikawa 1},~\cite{namikawa 2}:
% 
% \bigskip
% 
% \begin{theo}[Bellamy]
% Let $\XC$ be an affine Poisson variety which is a Poisson deformation of the 
% variety $\ZC_0=(V \times V^*)/W$. Then there exists $k \in \CM^\paramset$ such 
% that $\XC \simeq \ZC_k$ as a Poisson variety.
% \end{theo}
% 
\bigskip

% \subsubsection{Euler element}
% Let $(y_1,\dots,y_n)$ be a basis of $V$ and let $(x_1,\dots,x_n)$ 
% denote its dual basis. As in~\cite[??]{calogero}, we set
% $$\euler % = ~\sum_{j=1}^n x_j y_j + \sum_{s \in \Ref(W)} \e(s) c_k(s) s 
% =~\sum_{j=1}^n x_j y_j +
% \sum_{H\in\AC}\sum_{j=0}^{e_H-1}e_H\ k_{H,j}\varepsilon_{H,j}.$$
% Recall that $\euler$ does not depend on the choice of the basis of $V$. 
% Also
% \equat\label{eq:euler centre}
% \euler \in \Zb_k,\qquad \Frac(\Zb_k)=\Frac(\Pb)[\euler]
% \endequat
% and
% \equat\label{eq:euler poisson}
% \{\euler,z\}=d z 
% \endequat
% if $z \in \Zb_k$ is homogeneous of degree $d$ 
% (see for instance~\cite[??]{calogero}).\todo{Utile ?}
% 
% \bigskip

\subsubsection{Filtration} The tensor algebra $\Trm(V \oplus V^*)$ is naturally filtered 
by the subspaces $\bigl(\bigoplus_{j=0}^d (V \oplus V^*)^{\otimes j}\bigr)$. This 
induces a filtration of $\Trm(V \oplus V^*) \rtimes W$ by putting $W$ in degree $0$ and 
so induces a filtration $(\FC_j \Hb_k)_{j \geqslant 0}$ of the rational Cherednik algebra. 
By convention, we set $\FC_{-1} \Hb_k=0$. 
If $M$ is any subspace of $\Hb_k$, we set $\FC_j M=M \cap \FC_j\Hb_k$, so that 
$M$ also inherits a filtration and we denote by $\rees_\FC M$ the {\it Rees module} of $M$ 
(associated with the filtration $(\FC_j X)_{j \geqslant 0}$), namely the $\CM[\hbar]$-submodule  
of $\CM[\hbar] \otimes M$ equal to
$$\rees_\FC M=\bigoplus_{j \ge 0} \hbar^j \FC_j M.$$
Recall that, if $\l \in \CM$, then 
\equat\label{eq:rees-facile}
\CM[\hbar]/\langle \hbar - \l\rangle \otimes_{\CM[\hbar]} \rees_\FC M
\simeq
\begin{cases}
M & \text{if $\l \neq 0$,}\\
\grad_\FC(M) & \text{if $\l=0$,}\\
\end{cases}
\endequat
where $\grad_\FC(M)=\bigoplus_{j \ge 0} \FC_j M /\FC_{j-1} M$ is the graded vector 
space associated with $M$ and its filtration. 

If $A$ is a subalgebra of $\Hb_k$ and $J$ is an ideal of $A$, then 
$\rees_\FC(A)$ is a subalgebra of $\CM[\hbar] \otimes A$ (called the {\it Rees 
algebra} of $A$) and $\rees_\FC(J)$ is an ideal of $\rees_\FC(A)$. Recall~\cite[Theo.~1.3]{EG} that
\equat\label{eq:rees}
\grad_\FC\Hb_k \simeq \Hb_0 =\CM[V \times V^*] \rtimes W 
\qquad\text{and}\qquad \grad_\FC \Zb_k \simeq \Zb_0=\CM[V \times V^*]^W.
\endequat

\bigskip

\subsubsection{Action of the normalizer} 
The group $\Nrm_{\Gb\Lb_\CM(V)}(W)$ acts on the set $\paramset$ and so on 
the space of parameters $\CM^\paramset$. If $\t \in \Nrm_{\Gb\Lb_\CM(V)}(W)$, 
then $\t$ induces an isomorphism of algebras $\Hb_k \longto \Hb_{\t(k)}$. 
So, if $\t(k)=k$, then it induces an action on the algebra $\Hb_k$, 
and so on its center $\Zb_k$ and on the Calogero-Moser space $\ZC_k$, which preserves 
the $\CM^\times$-action and the Poisson bracket. We set
$$\d(\t)=\max_{w \in W} \dim V^{w\t}.$$
Of course, $\d(\t)$ depends only on the coset $W\t$ and not on $\t$. 
We say that $\t$ is {\it $W$-full} if $\d(\t)=\dim V^\t$. 
Since $W$ acts trivially on $\ZC_k$, the study of the action 
of $\t$ on $\ZC_k$ is equivalent to the study of the action of $w\t$. 
So, by replacing $\t$ by $w\t$ if necessary, we may assume that 
$\t$ is $W$-full.

% Finally, if $S$ is a subset of $\Nrm_{\Gb\Lb_\CM(V)}(W)$, we set $W(S)=W[V^{\langle S \rangle}]$. 
% In other words, $W(S)=W_{V^{\langle S \rangle}}^\setw/W_{V^{\langle S \rangle}}^\ptw$: the group 
% $W(S)$ acts faithfully on $V^{\langle S \rangle}$. 

\bigskip

\begin{exemple}\label{ex:regular}
An element $\t \in \Nrm_{\Gb\Lb_\CM(V)}(W)$ is called {\it $W$-regular} 
(or simply {\it regular} if $W$ is clear from the context) if $V^\t \cap V_\reg \neq \vide$. 
A $W$-regular element of $\Nrm_{\Gb\Lb_\CM(V)}(W)$ is $W$-full~\cite{springer}.\finl
\end{exemple}

\bigskip

\boitegrise{{\bf Hypothesis and notation.} 
{\it From now on, and until the end of this paper, we assume that
$$W=\langle \Ref(W) \rangle,$$
we fix a parameter $k \in \CM^\paramset$ 
and an element $\t$ of {\bfit finite} order of $\Nrm_{\Gb\Lb_\CM(V)}(W)$ such that $\t(k)=k$. 
We also assume that $\t$ is {\bfit $W$-full}. \\
\hphantom{AA} If $\bigstar$ is one of the objects 
defined in the previous sections ($\Hb_k$, $\ZC_k$, $\paramset$, $\AC$,\dots), 
we will sometimes denote it by $\bigstar(W)$ or $\bigstar(V,W)$ if we need 
to emphasize the context.}}{0.82\textwidth}

\bigskip

\section{Recollection about Poisson structures and symplectic leaves}

\medskip

\boitegrise{{\bf Notation.} {\it We fix in this section, and only in this section, 
a commutative noetherian Poisson $\CM$-algebra $R$, whose Poisson bracket 
is denoted by $\{,\}$.}}{0.75\textwidth}

\bigskip

\subsection{Poisson ideals} 
An ideal $I$ of $R$ is called a {\it Poisson ideal} of $R$ if $\{r,I\} \subset I$ 
for all $r \in R$. The following facts may be found in~\cite[Lem.~3.3.3]{dixmier}:

\bigskip

\begin{prop}\label{prop:poisson-basique}
Let $I$ be a Poisson ideal of $R$. Then:
\begin{itemize}
\itemth{a} Every minimal prime ideal containing $I$ is Poisson.

\itemth{b} The radical of $I$ is Poisson.
\end{itemize}
\end{prop}

\bigskip

% If $I$ is an ideal of $R$, we define its {\it Poisson core} $\CC(I)$ 
% to be the biggest Poisson ideal contained in $I$. If $I$ is prime, 
% it follows from Proposition~\ref{prop:poisson-basique}(a) that $\CC(I)$ 
% is also prime. The ideal $I$ is called {\it Poisson primitive} if it is 
% equal to $\CC(\mG)$ for some maximal ideal $\mG$ of $R$. In particular, 
% a Poisson primitive ideal is prime and Poisson. The converse holds in some 
% cases~\cite[Lem.~3.4]{brown gordon}.\todo{Utile ?}
% 
% \bigskip
% 
% \begin{prop}\label{prop:poisson-primitif}
% Assume that there are only finitely many Poisson primitive ideals of $R$. 
% Then every prime and Poisson ideal of $R$ is primitive.
% \end{prop}
% 
% \bigskip
% 
% \begin{proof}
% Let $\pG$ be a prime and Poisson ideal of $R$. Since $R$ is noetherian, 
% $\pG$ is the intersection of the maximal ideals containing $\pG$ (ref???). 
% Consequently, $\pG$ is the intersection of the Poisson primitive ideals 
% containing $\pG$. As there are only finitely many Poisson primitive ideals, 
% the result follows from~\cite[Chap.~II,~\S{1},~Prop.~1]{bourbaki}.
% \end{proof}
% 
% \bigskip
% 
% \begin{exemple}\label{ex:leaves}
% Let $\XC$ be an affine variety and assume here that $R=\CM[\XC]$ is endowed 
% with a Poisson bracket. Then Brown-Gordon~\cite{brown gordon} have defined a partition of 
% $\XC$ into symplectic leaves and symplectic leaves are in one-to-one correspondance 
% with Poisson primitive ideals of $\CM[\XC]$.
% \end{exemple}
% 
% 

\subsection{Normalisation} 
The next result is due to Kaledin~\cite{kaledin}:

\bigskip

\begin{theo}[Kaledin]\label{theo:kaledin}
Assume that $R$ is a domain. 
Then there is a unique Poisson bracket on the normalisation of $R$ extending $\{,\}$.
\end{theo}

\bigskip

\subsection{Action of a finite group}\label{sub:poisson-fixed}
We assume in this subsection that we are given a finite group $G$ acting 
on the $\CM$-algebra $R$ in such a way that the Poisson bracket is $G$-equivariant 
(i.e. $\{g(r),g(r')\}=g(\{r,r'\})$ for all $g \in G$ and $r$, $r' \in R$). 
Let $I$ denote the ideal of $R$ generated by the family 
$(g(r)-r)_{\substack{g \in G \\ r \in R}}$. Then $R/I$ is the biggest 
quotient algebra of $R$ on which $G$ acts trivially.

Since $G$ is finite and $\CM$ has characteristic $0$, the natural map
$$R^G \longto (R/I)^G=R/I$$
is surjective and its kernel is $I^G$. Moreover, $R^G$ is a Poisson 
subalgebra of $R$ (because the Poisson bracket is $G$-equivariant). 
Note that $I$ is not in general a Poisson ideal of $R$, but 
it is easily checked that 
\equat\label{eq:ideal-fixe}
\text{\it $I^G$ is a Poisson ideal of $R^G$.}
\endequat
Therefore, $R/I=R^G/I^G$ can be naturally endowed with a Poisson bracket.
And, by Proposition~\ref{prop:poisson-basique}(b), $R/\sqrt{I}=R^G/\sqrt{I^G}$ 
also inherits a Poisson bracket.

\bigskip

\begin{rema}\label{rem:poisson variety}
If $R=\CM[\XC]$ is the coordinate ring of an affine variety $\XC$, then $R/I$ 
is the coordinate ring of the $G$-fixed points scheme of $\XC$ (which will be denoted by 
$\XC^{(G)}$), while $R/\sqrt{I}$ is the coordinate ring of its reduced subscheme 
(which will be denoted by $\XC^G$). The above construction shows that 
the closed subvariety $\XC^G$ of $\XC$ inherits a Poisson structure from the one on $\XC$, 
even though it is not in general a Poisson subvariety of $\XC$ 
(i.e. the natural map $\XC^G \longinjto \XC$ is not Poisson). 
However, $\XC/G$ is also a Poisson variety and the natural map $\XC^G \longinjto \XC/G$ 
is Poisson, i.e. $\XC^G$ is a closed Poisson subvariety of $\XC/G$.

If moreover $\XC$ is smooth, then $\XC^{(G)}=\XC^G$ is also smooth, and if the Poisson 
structure on $\XC$ makes it into a symplectic variety, then $\XC^G$ 
is also symplectic for the induced Poisson structure.\finl
\end{rema}

\bigskip

\begin{exemple}\label{ex:sous-espace}
Let $E$ be a $\CM$-vector space endowed with a symplectic form $\o$ 
and assume here that $R=\CM[E]$ and that $G \subset \Sb\pb(E,\o)$. 
Then the restriction of $\o$ to $E^G$ is non-degenerate, so this 
endows $E^G$ with a structure of Poisson (even more, symplectic) 
variety. On the other hand, via the above Remark~\ref{rem:poisson variety}, 
the variety $E^G$ also inherits from $E$ a structure of Poisson variety. 
It is easily checked that both structures coincide.\finl
\end{exemple}

\bigskip

\subsection{Symplectic leaves}\label{sub:symp-general}
Assume in this subsection that $R=\CM[\XC]$ is the coordinate ring of an affine variety $\XC$. 
Then Brown-Gordon~\cite{brown gordon} have defined a stratification of $\XC$ by  
{\it symplectic leaves}, which are in general not algebraic subvarieties of $\XC$. 
We denote by $\symp(\XC)$ the set of symplectic leaves of $\XC$. 

When $\XC$ has finitely many symplectic leaves, then the symplectic 
leaves are algebraic~\cite[Prop.~3.7]{brown gordon} and the stratification of $\XC$ into symplectic 
leaves is given as follows. Let $(\SC_j)_{j \ge 0}$ be the sequence of 
closed subvarieties of $\XC$ defined by
$$
\begin{cases}
\SC_0=\ZC_k^\t,\\
\text{If $j \ge 0$, then $\SC_{j+1}$ is the reduced singular locus of $\SC_j$.}
\end{cases}
$$
Then the symplectic leaves of $\XC$ are the irreducible components of the locally 
closed subvarieties $(\SC_j \setminus \SC_{j+1})_{j \ge 0}$. 
Let $\pspec(\CM[\XC])$ denote the set of prime ideals that are Poisson. 
If $\SC$ is a symplectic leaf of $\XC$, we denote by $\pG_\SC$ 
the defining ideal of $\overline{\SC}$ in $\CM[\XC]$: it belongs to $\pspec(\CM[\XC])$. 
If $\XC$ has finitely many symplectic leaves, then the map 
\equat\label{eq:poisson prime general}
\fonctio{\symp(\XC)}{\pspec(\CM[\XC])}{\SC}{\pG_\SC}
\endequat
is bijective~\cite[Lem.~3.4]{brown gordon}. The inverse is given as follows: 
if $\pG \in \pspec(\CM[\XC])$ corresponds to $\SC$ through this bijection, 
then $\SC$ is the smooth locus of the closed irreducible subvariety of $\XC$ 
defined by $\pG$. 
% In this case, if $x \in \XC$ corresponds to the maximal 
% ideal $\mG$ of $\CM[\XC]$, and if $\SC$ denotes the symplectic leaf contining $x$, then 
% $\pG_\SC=\CC(\mG)$.

\bigskip

\begin{lem}\label{lem:ferme-leaves}
Assume that $\XC$ has finitely many symplectic leaves and that $\YC$ is a locally 
closed Poisson subvariety of $\XC$. Then $\YC$ has finitely many symplectic leaves.
\end{lem}

\bigskip

\begin{proof}
Taking the closure of $\YC$, which is also Poisson, allows to assume that $\YC$ 
is closed. Let $(\SC_l)_{l \in L}$ be the family of symplectic leaves of $\XC$ 
(for some finite indexing set $L$). Let $\IC$ be an irreducible component of $\YC$. 
Then $\IC$ is also Poisson by Proposition~\ref{prop:poisson-basique}(a) so 
it is the closure of a symplectic leaf thanks to the bijection~\eqref{eq:poisson prime general}. 
In particular, there exists a subset $I$ of $L$ such that $\IC$ is the union of 
the $\SC_i$, for $i \in I$. This proves that $\YC$ is a union of symplectic 
leaves of $\XC$, each of which being also a symplectic leaf of $\YC$.
\end{proof}

% In particular, 
% if $G$ is a connected algebraic group acting on $\XC$ and preserving the 
% Poisson bracket, then
% \equat\label{eq:g-stable}
% \text{\it Every symplectic leaf is $G$-stable.}
% \endequat

\bigskip

Now, let $G$ be a finite group acting on $\XC$ and preserving the Poisson 
bracket. Then $\XC/G$ is an affine Poisson variety (because $\CM[\XC/G]=\CM[\XC]^G$ 
is a Poisson subalgebra of $\CM[\XC]$, see Remark~\ref{rem:poisson variety}). 
If $H$ is a subgroup of $G$, we denote by $\XC(H)$ the set of elements 
$x \in \XC$ whose stabilizer is exactly $H$. Then $\XC(H)$ is a locally closed subvariety 
of $\XC$ (it is open in $\XC^H$). The subgroup $H$ is called {\it parabolic} if $\XC(H) \neq \vide$. 
Let $\para(G)$ denote the set of parabolic subgroups of $G$. 

If $\HG$ is a conjugacy class of parabolic subgroups of $G$, we denote by $(\XC/G)(\HG)$ 
the image of $\XC(H)$ in $\XC/G$ for some (or any) $H \in \HG$. Then the group 
$\Nrmov_G(H)/H$ acts freely on $\XC(H)$ and the natural map 
$\XC(H) \to (\XC/G)(\HG)$ induces an isomorphism 
\equat\label{eq:iso-normalisateur}
\XC(H)/G(H) \longiso (\XC/G)(\HG).
\endequat
Indeed, if $g \in G$ and $x$, $x' \in \XC(H)$ are such that $g\cdot x = x'$, 
then $H=G_{x'}=\lexp{g}{G_x}=\lexp{g}{H}$ and so $g \in \Nrm_G(H)$. 
The next results generalizes slightly~\cite[Prop.~7.4]{brown gordon}:

\bigskip

\begin{prop}\label{prop:leaves-lisse}
Assume that $\XC$ is smooth and symplectic. Then the symplectic leaves of 
$\XC/G$ are the irreducible components of the locally closed subvarieties 
$(\XC/G)(\HG)$ where $\HG$ runs over $\para(G)/G$. 

In particular, if all the subvarieties $(\XC/G)(\HG)$ are irreducible, then 
$$\XC/G = \dot{\bigcup_{\HG \in \para(G)/G}} (\XC/G)(\HG)$$
is the stratification of $\XC/G$ into symplectic leaves.
\end{prop}

\bigskip

\begin{proof}
Let $\HG \in \para(G)$ and let $H \in \HG$. 
Since $\XC$ is smooth and symplectic, the subvariety $\XC^H$ is also smooth and symplectic. 
So its open subset $\XC(H)$ is also smooth and symplectic as well as $(\XC/G)(\HG)$ 
thanks to the isomorphism~\eqref{eq:iso-normalisateur}. 
And the morphism $\XC^H \to \XC/G$ is Poisson: this proves that any irreducible 
component of $(\XC/G)(\HG)$ is contained in a unique symplectic leaf. In particular, 
$\XC/G$ has finitely many symplectic leaves.

It remains to show that any irreducible component $\IC$ of $(\XC/G)(\HG)$ is 
a symplectic leaf. But $\ICov$ is a closed Poisson subvariety of 
$\XC/G$, so its smooth locus is a symplectic leaf of $\XC/G$ by the 
bijection~\eqref{eq:poisson prime general}. Since $\IC$ is smooth, 
it remains to show that $\IC$ is the smooth locus of $\ICov$. 
But, by the bijection~\eqref{eq:poisson prime general} applied to $\ICov$, 
this follows from the fact that $\ICov \setminus \IC$ is a closed Poisson 
subvariety of $\ICov$ (indeed, $\ICov \setminus \IC$ is the intersection 
of $\ICov$ with the union of the $\overline{(\XC/G)(\HG')}$, where 
$\HG'$ runs over the set of conjugacy classes of parabolic subgroups 
of $G$ containing strictly at least one element of $\HG$).
\end{proof}

\bigskip

\begin{coro}\label{coro:finitude}
Assume that $\XC$ has finitely many symplectic leaves. Then 
$\XC/G$ and $\XC^G$ have finitely many symplectic leaves.
\end{coro}

\bigskip

\begin{proof}
Let $\SC$ denote a symplectic leaf of $\XC$ and let $H=G_\SC^\setw$. 
As symplectic leaves form a partition of $\XC$, and since 
$g(\SC)$ is a symplectic leaf of $\XC$ for any $g \in G$, we get that 
$$g(\SC) \cap \SC = \vide$$
for all $g \not\in H$. So, the image $\SC$ in $\XC/G$ is isomorphic 
to $\SC/H$ and is a locally closed Poisson subvariety of $\XC/G$. 
But, by Proposition~\ref{prop:leaves-lisse}, $\SC/H$ has finitely 
many symplectic leaves. 

As $\XC$ has finitely many symplectic leaves, this shows that $\XC/G$ 
also has finitely many symplectic leaves.
Now, $\XC^G$ is a closed Poisson subvariety of $\XC/G$, so 
it also admits finitely many symplectic leaves by Lemma~\ref{lem:ferme-leaves}. 
\end{proof}

\bigskip

As a consequence of the above proof, we get:

\bigskip

\begin{coro}\label{coro:finitude-bis}
Assume that $\XC$ has finitely many symplectic leaves and that $G$ acts freely 
on $\XC$. Then the map 
$\symp(\XC)/G \longto \symp(\XC/G)$ sending the $G$-orbit of 
a symplectic leaf of $\XC$ 
to its image in $\XC/G$ is well-defined and bijective.
\end{coro}

\bigskip

% \begin{lem}\label{lem:finitude}
% Assume that $\XC$ has finitely many symplectic leaves. Then 
% $\XC^G$ has finitely many symplectic leaves.
% \end{lem}
% 
% \bigskip
% 
% \begin{proof}
% We have
% $$\XC^G=\dot{\bigcup_{\SC \in \symp(\XC)^G}} \SC^G.$$
% But, if $\SC \in \symp(\XC)^G$, then $\SC$ is smooth and symplectic, 
% and so $\SC^G$ is smooth and symplectic. Therefore, all the irreducible 
% components of $\SC^G$ are contained in a unique symplectic leaf, 
% and the finiteness of $\symp(\XC^G)$ follows.
% \end{proof}
% 
% 
% \bigskip

\section{Recollection of Lehrer-Springer theory}\label{sub:springer}

\medskip

\subsection{Reflection groups}
Note that~\cite{springer} 
\equat\label{eq:delta-tau}
\d(\t)=\dim (V/W)^\t.
\endequat
Therefore, since $(V/W)^\t$ is irreducible 
(it is isomorphic to an affine space~\cite{springer}), we get that
\equat\label{eq:w-full}
\text{\it the natural map $V^\t \to (V/W)^\t$ is onto.}
\endequat
For simplifying notation, we set $W_\t=W_{V^\t}^\setw/W_{V^\t}^\ptw$. Note
that $W^\t \subset W_{V^\t}^\setw$. Moreover, $W_\t$ acts faithfully on $V^\t$, so 
\equat\label{eq:tau trivial}
\text{\it $\t$ acts trivially on $W_\t$.}
\endequat
Lehrer-Springer theory~\cite[Theo.~2.5~and~Cor.~2.7]{lehrer springer} 
gives the following result:
\def\ii{i\hskip-0.3mm i}

\bigskip

\begin{theo}[Springer, Lehrer-Springer]\label{theo:lsp}
Recall that $\t$ is $W$-full. Then:
\begin{itemize}
\itemth{a} The group $W_\t$ is a reflection group 
for its action on $V^{\t}$.
 
\itemth{b} The natural map 
$$i_\t : V^\t/W_\t \longto (V/W)^\t$$
is an isomorphism of varieties.

\itemth{c} The reflecting hyperplanes of $W_\t$ are exactly the intersections 
with $V^\t$ of the reflecting hyperplanes of $W$ which do not contain $V^\t$. 
\end{itemize}
\end{theo}

\bigskip

Similarly, the natural map $i_\t^\ve : V^{*\t}/W_\t \to (V^*/W)^\t$ is an isomorphism of 
varieties.

\bigskip

\begin{exemple}\label{ex:regular-w}
If $\t$ is $W$-regular (as defined in Example~\ref{ex:regular}), 
then $W_\t=W^\t$ by~\cite{springer}.\finl
\end{exemple}

\bigskip

\def\opp{{\mathrm{op}}}

\subsection{$\taub$-split parabolic subgroups}
A parabolic subgroup $P$ of $W$ is called {\it $\t$-split} 
if it is the stabilizer of some point of $V^\t$ (i.e. if $P=W_{V^P \cap V^\t}^\ptw$ or, 
in other words, if $\VC(P) \cap V^\t  \neq \vide$). This is equivalent 
to say that $P$ is the stabilizer of some point of $V^*$. Note the following easy fact:
\equat\label{eq:inter split}
\text{\it The intersection of $\t$-split parabolic subgroups is $\t$-split.}
\endequat

In this case, $P$ is normalized by $\t$ and $\t$ is $P$-full, and 
we define the {\it $\t$-rank} of $P$ to be the number $\dim(V^P)^\t$. 
We denote by $\para_\t(W)$ the set of $\t$-split parabolic subgroups 
of $W$. If $P \in \para_\t(W)$, then $W_{V^\t}^\ptw \subset P$ and 
$P_\t=(P \cap W_{V^\t}^\setw)/W_{V^\t}^\ptw$ is a parabolic subgroup 
of $W_\t$. This shows that the map 
$$\fonctio{\para_\t(W)}{\para(W_\t)}{P}{P_\t}$$
is well-defined. 

\bigskip

\begin{lem}\label{lem:p bijectif}
The map
$$\fonctio{\para_\t(W)}{\para(W_\t)}{P}{P_\t}$$
is bijective.
\end{lem}

\bigskip

\begin{proof}
First, if $Q$ is a parabolic subgroup of $W_\t$, then 
there exists $v \in V^\t$ such that $Q=(W_\t)_v$ and so, if we set $P=W_v$, then 
$P$ is $\t$-split and $P_\t=Q$. This shows that the map is surjective. 

\medskip

Now, if $P$ is a $\t$-split parabolic subgroups of $W$, 
then 
\equat\label{eq:vptau}
(V^P)^\t = (V^\t)^{P_\t}.
\endequat
\begin{quotation}
\begin{proof}[Proof of~\eqref{eq:vptau}]
Since $P$ is $\t$-split, there exists $v \in (V^P)^\t$ such that $P=W_v$ 
(and so $P_\t=(W_\t)_v$). Therefore, 
$$(V^\t)^{P_\t} = \bigcap_{\substack{H \in \AC(V^\t,W_\t) \\ v \in H}} H,$$
and so, by Theorem~\ref{theo:lsp}(c), 
$$(V^\t)^{P_\t} = V^\t \cap \Bigl(\bigcap_{\substack{H \in \AC(V,W) \\ v \in H}} H \Bigr)=V^\t \cap V^P,$$
as expected.
\end{proof}
\end{quotation}
Since $P=W_{(V^P)^\t}^\ptw$, the group $P_\t$ determines $P$. This means 
that the map of the lemma is injective.
\end{proof}

\bigskip

If $\QG \in \para(W_\t)/W_\t$ and $Q \in \QG$, we denote by $\VC_\t(Q)$ and $\UC_\t(\QG)$ 
the analogues of $\VC(P)$ and $\UC(\PG)$ for $\PG \in \para(W)/W$ and $P \in \PG$. 
We also define similarly $\VC_\t^*(Q)$ and $\UC_\t^*(\QG)$. 
The same argument as in the above proof (using Theorem~\ref{theo:lsp}(c)) 
shows that, if $P$ is $\t$-split, then 
\equat\label{eq:v split}
\VC_\t(P_\t)=\VC(P)^\t.
\endequat

\subsection{Normalizers}
Fix a $\t$-split parabolic subgroup $P$ of $W$. If $w \in W_\t$, then $\lexp{w}{P}$ 
does not depend on the representative of $w$ in $W_{V^\t}^\setw$, because 
$W_{V^\t}^\ptw \subset P$ by definition. So we can define the {\it normalizer} 
$\Nrm_{W_\t}(P)$ of $P$ in $W_\t$ and, by the bijectivity proved in Lemma~\ref{lem:p bijectif}, 
it coincides with the normalizer $\Nrm_{W_\t}(P_\t)$. The kernel 
of the well-defined composition
$$\Nrm_{W_\t}(P_\t)=\Nrm_{W_\t}(P) \longto \Nrmov_W(P)/P$$
is equal to $P_\t$, so we get a natural injective map
\equat\label{eq:injective map}
\Nrmov_{W_\t}(P_\t) \longinjto \Nrmov_W(P).
\endequat
Now, $\t$ acts on $\Nrmov_W(P)$. The next result describes the image of 
the above injective map:

\bigskip

\begin{lem}\label{lem:normalisateurs}
The image of the morphism~\eqref{eq:injective map} is equal to 
$\Nrmov_W(P)^\t$.
\end{lem}

\bigskip

\begin{proof}
Let $G$ denote the image of the morphism~\eqref{eq:injective map} 
and let $v \in \VC(P)^\t$ (so that $P=W_v$).

Let $w \in \Nrm_{W_\t}(P_\t)$ and let $\wdo$ be a representative 
of $w$ in $W_{V^\t}^\setw$. Then $\wdo(v) \in \VC(P)^\t$ by~\eqref{eq:v split}. 
So $\t(\wdo(v))=\wdo(v)$, i.e. $\wdo^{-1} \t(\wdo) \in P$. So the image of $w$ in 
$\Nrmov_W(P)$ is $\t$-invariant. This proves that $G \subset \Nrmov_W(P)^\t$. 

Conversely, let $w \in \Nrmov_W(P)^\t$ and let $\wdo$ denote a representative 
of $w$ in $\Nrm_W(P)$. Then $\t(\wdo(v))=\wdo (\wdo^{-1}\t(\wdo))(v)$. 
But $(\wdo^{-1}\t(\wdo))(v)=v$ since $\wdo^{-1}\t(\wdo) \in P$ by hypothesis. 
So $v$ and $\wdo(v)$ belong to $V^\t$ so, by Lehrer-Springer Theorem~\ref{theo:lsp}(b), 
there exists $x \in W_\t^\setw$ such that $\wdo(v)=x(v)$. In other words, 
$x^{-1}\wdo \in P$. Moreover, $v$ and $\wdo(v)$ both belong to $\VC(P)$, 
so $x$ normalizes $P$ (and $P_\t$), so $w$ is the image of $x$ under 
the morphism~\eqref{eq:injective map}. In other words, $\Nrmov_W(P)^\t \subset G$.
\end{proof}

\bigskip

Thanks to Lemma~\ref{lem:normalisateurs}, we will identify $\Nrmov_{W_\t}(P_\t)$ 
with $\Nrmov_W(P)^\t$. Note that $\Nrmov_{W_\t}(P_\t)=\Nrmov_W(P)^\t$ is the 
stabilizer of the set $\VC(P)^\t$ in $\Nrmov_W(P)$.

\bigskip

\subsection{Orbits of ${\boldsymbol{\tau}}$-split parabolic subgroups} 
We denote by $\para(W)^\t_\split$ the set of $\t$-split parabolic subgroups of $W$ and 
by $(\para(W)/W)^\t_\split$ the set of $W$-orbits of parabolic subgroups 
of $W$ containing a $\t$-split one. The group $W_{V^\t}^\setw$ acts on $\para(W)^\t_\split$ 
by conjugacy and, since any $\t$-split parabolic subgroup of $W$ contains $W_{V^\t}^\ptw$, 
this action factorizes through an action of $W_\t$. 
If $\PG \in (\para(W)/W)^\t_\split$, we 
set $\PG^\t_\split = \PG \cap \para(W)_\split^\t$. 
Now, let $\EC_P$ (resp. $\ECt_P$) denote the set of elements $w \in \Nrm_W(P)$ (resp. $\Nrmov_W(P)$) 
such that $\VC(P)^{w\t} \neq \vide$. Then $\Nrmov_W(P)$ acts by conjugacy on the set 
$\ECt_P\t$. If $w \in \EC_P$, we denote 
by $[w\t]$ the $\Nrmov_W(P)$-orbit of the image of $w\t$ in $\Nrmov_W(P)\t$. 

\bigskip

\begin{prop}\label{prop:param-tau-split}
Let $P$ be a $\t$-split parabolic subgroup and let $\PG$ denote its $W$-orbit. Then:
\begin{itemize}
\itemth{a} Let $x \in W$. Then $\lexp{x}{P}$ is $\t$-split if and only if 
$x^{-1}\t(x) \in \EC_P$. 

\itemth{b} The map $\PG_\split^\t \to \ECt_P\t/\Nrmov_W(P)$, $\lexp{x}{P} \mapsto [x^{-1}\t x]$ 
is well-defined and induces a bijection
$$\PG_\split^\t/W_\t \longiso \ECt_P\t/\Nrmov_W(P).$$
\end{itemize}
\end{prop}

\bigskip

\begin{proof}
(a) Assume that $\lexp{x}{P}$ is $\t$-split. In other words, there exists $v \in V^\t$ 
such that $\lexp{x}{P}=W_v$. Now, let $w=x^{-1}\t(x)$: then $x^{-1}(v) \in \VC(P)^{w\t}$ 
and so $w \in \EC_P$.

Conversely, assume that $w=x^{-1}\t(x) \in \EC_P$. Then there exists $v \in \VC(P)$ 
such that $v \in V^{w\t}$. Therefore, $P=W_v$ and so $\lexp{x}{P}=W_{x(v)}$. 
But $\t(x(v))=xx^{-1}\t(x)\t(v)=xw\t(v)=x(v)$, so $x(v) \in V^\t$. This implies 
that $\lexp{x}{P}$ is $\t$-split by definition.

\medskip

(b) Let us first show that the map (let us denote it by $\phi$) 
is well-defined. For this, let $x$ and $y$ be two elements 
of $W$ such that $\lexp{x}{P}=\lexp{y}{P}$ is $\t$-split. Then there exists $u \in \Nrm_W(P)$ 
such that $y=xu$. So $y^{-1}\t y = u^{-1} x^{-1} \t x u$ and so $[y^{-1}\t y]=[x^{-1}\t x]$, 
as expected. 

Let us now prove that $\phi$ is constant on $W_\t$-orbits. For this, let $w \in W_\t$ 
and $x \in W$ be such that $\x^{-1}\t(x) \in \EC_P$. Then 
$(xw)^{-1}\t(wx)=x^{-1} w^{-1}\t(w) x x^{-1}\t(x)$. 
But $w^{-1}\t(w) \in W_{V^\t}^\ptw \subset P$ by~\eqref{eq:tau trivial}, 
so the images of $x^{-1}\t(x)$ and $(wx)^{-1}\t(wx)$ in $\ECt_P$ coincide. 
Therefore, $\phi$ actorizes through a map 
$$\phit : \PG_\split^\t/W_\t \longto \EC_P\t/\Nrmov_W(P).$$

Let us prove that $\phit$ is injective. So let $x$ and $y$ be two elements of $W$ such that 
$\lexp{x}{P}$ and $\lexp{y}{P}$ are $\t$-split and $[x^{-1}\t x]=[y^{-1}\t y]$. 
Then there exists $u \in \Nrm_W(P)$ and $p \in P$ such that 
$y^{-1}\t y = u^{-1}x^{-1}\t x u p$. In particular,
$\VC(P)^{y^{-1}\t y}=\VC(P)^{u^{-1}x^{-1}\t x u}$. Since $\lexp{x}{P}=\lexp{xu}{P}$, 
we may (and we will) assume that $u=1$. As $\lexp{x}{P}$ is $\t$-split, the set 
$\VC(\lexp{x}{P})^\t$ is non-empty, so we may pick an element $v \in \VC(\lexp{x}{P})^\t$. 
Then
$$\t yx^{-1}(v)=yy^{-1}\t y x^{-1}(v)=yx^{-1}\t x p x^{-1}(v).$$
But $x^{-1}(v) \in \VC(P)$ so $px^{-1}(v)=x^{-1}(v)$. Consequently,
$\t yx^{-1}(v)= yx^{-1} \t(v)=yx^{-1}(v)$. In other words, 
$yx^{-1}(v) \in \VC(\lexp{y}{P})^\t \subset V^\t$. By Lehrer-Springer Theorem~\ref{theo:lsp}(b), 
there exists $a \in W_\t$ such that $yx^{-1}(v)=a(v)$. Then 
$$\lexp{y}{P}=\lexp{yx^{-1}}{(\lexp{x}{P})}=\lexp{yx^{-1}}{W_v}
=W_{yx^{-1}(v)}=W_{a(v)}=\lexp{a}{W_v}=\lexp{a}{(\lexp{x}{P})},$$
which shows that $\lexp{y}{P}$ and $\lexp{x}{P}$ are $W_\t$-conjugate.

Let us now prove that $\phit$ is surjective. So let $w \in \EC_P$. Then there exists 
$v \in \VC(P)^{w\t}$. So $W \cdot v \in (V/W)^\t$. By Lehrer-Springer Theorem~\ref{theo:lsp}, 
there exists $x \in W$ such that $x(v) \in V^\t$. Therefore, 
$\lexp{x}{P}=W_{x(v)}$ is $\t$-split and $v \in V^{x^{-1}\t x}$. 
So, if we set $p=w^{-1}x^{-1}\t(x)$, then $p(v)=v$ so $p \in P$ and 
$\phi(\lexp{x}{P})=[x^{-1}\t(x)]=[wp]=[w]$, as desired.
\end{proof}

\bigskip

\subsection{Stratification of ${\boldsymbol{(V/W)^\t}}$}\label{sub:strat}
Applying~\ref{sub:para} to the pair $(V^\t,W_\t)$, 
the variety $V^\t/W_\t$ admits a stratification $(\UC_\t(\QG))_{\QG \in \para(W_\t)/W_\t}$
while the variety $(V/W)^\t$ admits a stratification $(\UC(\PG)^\t)_{\PG \in \para(W)/W}$. 
Both varieties are isomorphic and so both stratifications can be compared: through this 
isomorphism, the first one is a refinement of the second one, as will be shown in 
Corollary~\ref{coro:stratification} below by using 
Proposition~\ref{prop:param-tau-split}.

\bigskip

\begin{prop}\label{prop:stratification}
Let $\PG \in \para(W)/W$. Then $\UC(\PG)^\t$ is non-empty if and only if $\PG$ contains 
a $\t$-split parabolic subgroup.
\end{prop}

\bigskip

\begin{proof}
If $\PG$ contains a $\t$-split parabolic subgroup $P$ and if $v \in V^\t$ is such that 
$P=W_v$, then the $W$-orbit of $v$ belongs to $\UC(\PG)^\t$ which is therefore non-empty. 
Conversely, if $\UC(\PG)^\t$ is non-empty, it then follows from Theorem~\ref{theo:lsp}(b) 
that there exists $v \in V^\t$ whose $W$-orbit belongs to $\UC(\PG)^\t$. By construction, 
$W_v \in \PG$ and is $\t$-split.
\end{proof}

\bigskip

After eliminating the empty pieces, Proposition~\ref{prop:stratification} shows that 
$(V/W)^\t$ admits a stratification $(\UC(\PG)^\t)_{\PG \in (\para(W)/W)_\split^\t}$. 
Let us decompose the pieces of this stratification into irreducible components. 
For this, fix a $\t$-split parabolic subgroup $P$ and let $\PG$ 
denote its conjugacy class. Then $\UC(\PG)^\t$ is smooth since 
$\UC(\PG)$ is smooth and $\t$ has finite order, and we have
$$\UC(\PG)^\t = (\VC(P)/\Nrmov_W(P))^\t= \bigl(\bigcup_{w \in \ECt_P} \VC(P)^{w\t}\bigr)/\Nrmov_W(P).$$
By definition of $\VC(P)$, $\VC(P)^{w\t} \cap \VC(P)^{w'\t}=\vide$ is $w \neq w'$. 
If $E$ is a subset of $\ECt_P \t$, we denote by $\VC(P)^E$ the (disjoint) union of 
the $\VC(P)^g$ for $g \in E$. Then 
\equat\label{eq:uptau}
\UC(\PG)^\t =  \bigcup_{E \in \ECt_P\t/\Nrmov_W(P)} \VC(P)^E/\Nrmov_W(P).
\endequat
% Now, if $E \in \ECt_P\t/W(P)$, we denote by $\PG_E$ the associated $W_\t$-orbit 
% of $\t$-split parabolic subgroups through the bijection of Proposition~\ref{prop:param-tau-split}(b). 
% Let $(\PG_E$ denote the $W_\t$-conjugacy class of parabolic subgroups of $W_\t$ corresponding 
% to $\PG_E$ through the bijection of Lemma~\ref{lem:p bijectif}. 
Then $\VC(P)^E/\Nrmov_W(P)$ is the image 
of some $\VC(P)^g$ for some $g \in E$ and so $\VC(P)^E/\Nrmov_W(P)$ is closed (in $\UC(\PG)$) 
and irreducible. So the decomposition~\eqref{eq:uptau} is the decomposition 
of $\UC(\PG)^\t$ into irreducible (i.e. connected because they are disjoint) 
components.

So the stratification $(\UC(\PG)^\t)_{\PG \in (\para(W)/W)_\split^\t}$ of $(V/W)^\t$ 
together with the decomposition~\eqref{eq:uptau} provides a finer stratification 
of $(V/W)^\t$, indexed by the $W_\t$-orbits of $\t$-split parabolic subgroups 
(by using the bijection of Proposition~\ref{prop:param-tau-split}(b)). 
On the other hand, $V^\t/W_\t$ admits a stratification $(\UC_\t(\QG))_{\QG \in \para(W_\t)/W_\t}$. 
Both stratifications coincide through the isomorphism $i_\t$, as shown by the next result:

\bigskip

\begin{coro}\label{coro:stratification}
Let $\PG \in (\para(W)/W)_\split^\t$, let $P \in \PG$ and let $E \in \ECt_P\t/\Nrmov_W(P)$. 
Let $\PG_E$ denote the $W_\t$-orbit of $\t$-split parabolic subgroups of $W$ 
associated with $E$ through the bijection of Proposition~\ref{prop:param-tau-split}(b). 
Let $\QG_E$ denote the $W_\t$-orbit of parabolic subgroups of $W_\t$ of the form 
$Q_\t$ for $Q \in \PG_E$ (see Lemma~\ref{lem:p bijectif}). 
Then 
$$i_\t(\UC_\t(\QG_E))=\VC(P)^E/\Nrmov_W(P).$$
\end{coro}

\bigskip

\begin{proof}
Let $g \in E$ and let $x \in W$ be such that $[x^{-1}\t x] = [g]$ (the existence 
of such an $x$ is guaranteed by Proposition~\ref{prop:param-tau-split}(b)). 
We set $Q=\lexp{x}{P}$. Then $Q$ is $\t$-split by Proposition~\ref{prop:param-tau-split}(a) 
and $\PG_E$ (resp. $\QG_E$) is the $W_\t$-orbit of $Q$ (resp. $Q_\t$) by construction. 

Now, $i_\t(\UC_\t(\QG_E))$ is the image of $\VC(\lexp{x}{P})^\t$ in $(V/W)^P$ and, 
through the isomorphism $\VC(P)/\Nrmov_W(P) \simeq \UC(\PG)$, the result comes from 
the fact that $x^{-1}$ induces an isomorphism between $\VC(\lexp{x}{P})^\t$ 
 and $\VC(P)^{x^{-1}\t x}$.
\end{proof}

\bigskip

% \subsection{Completion} 
% Let $v \in V$. 
% We denote by $\mG(v)$ (or $\mG_{V,W}(v)$ if necessary) the ideal of $\CM[V/W]$ consisting 
% of functions of $\CM[V/W]=\CM[V]^W$ vanishing at $v$. As $\CM[V/W] \subset \Zb_k$, 
% we may define the ideal $\mG^{\Zb_k}(v)$ (or $\mG_{V,W}^{\Zb_k}(v)$) of $\Zb_k$ 
% generated by $\mG(v)$. Now, let $\Zbh_k^v$ (or $\Zbh_k^v(V,W)$) denote 
% the completion of $\Zb_k$ at $\mG^{\Zb_k}(v)$: by definition, 
% $$\Zbh_k^v = \varprojlim_j \Zb_k/\mG^{\Zb_k}(v)^j.$$ 
% By~\cite[Lem.~3.5(1)]{bellamy cuspidal}, 
% $\Zbh_k^v$ inherits a Poisson structure. 
% Recall~\cite[Coro.~5.4]{eisenbud} that the natural map $\Zb_k \longto \Zbh_k^v$ is injective, 
% so we will identify $\Zb_k$ with its image in $\Zbh_k^v$ and so view it 
% as a subalgebra of $\Zbh_k^v$.
% 
% In~\cite[Theo.~4.5]{bellamy cuspidal}, Bellamy constructs an isomorphism of Poisson algebras 
% \equat\label{eq:iso-completion}
% \iota_v : \Zbh_k^v(V,W) \longiso \Zbh_k^0(V,W_v).
% \endequat
% If moreover $\t \in \Nrm_{\Gb\Lb_\CM(V)}(W)$ is such that 
% $\t(v)=v$, then $W_v$ is $\t$-stable, the algebras $\Zbh_k^v(V,W)$ 
% and $\Zbh_k^0(V,W_v)$ inherits an action of $\t$, and the isomorphism $\iota_v$ 
% is $\t$-equivariant.
% 
% \bigskip

\section{The problem, the main result}\label{sec:symplectic}

\medskip

\subsection{Symplectic leaves}\label{sub:leaves}
Let $I_k$ denote the ideal of $\Zb_k$ generated by $(\t(z)-z)_{z \in \Zb_k}$. It 
is $\t$-stable. Recall from Remark~\ref{rem:poisson variety} that 
$\CM[\ZC_k^\t] = \Zb_k /\sqrt{I_k}=\Zb_k^\t /\sqrt{I_k^\t}$, 
and that $\Zb_k/I_k$ inherits a Poisson bracket which makes $\ZC_k^\t$ into 
an affine Poisson variety. Therefore, $\ZC_k^\t$ 
admits a stratification into symplectic leaves~\cite[\S{3.5}]{BG}. We denote by 
$\symp(\ZC_k^\t)$ the set of symplectic leaves of $\ZC_k^\t$. 

\bigskip

\begin{rema}\label{rem:fixe}
Note that $\ZC_k^\t$ is generally not irreducible, not connected, 
not equidimensional and that its irreducible components might 
not coincide with its connected components.\finl
\end{rema}

\bigskip

Since $\ZC_k$ has finitely many symplectic leaves~\cite[Prop.~7.4]{BG}, it follows from 
Corollary~\ref{coro:finitude} that $\ZC_k^\t$ has finitely many symplectic leaves too. 
They are obtained as in~\S\ref{sub:symp-general}. 

\bigskip

\begin{rema}\label{rem:symp-stable}
This description shows that the symplectic leaves of $\ZC_k^\t$ 
are $\CM^\times$-stable.\finl
\end{rema}

\bigskip

If $\SC$ is a symplectic leaf of $\ZC_k^\t$, we denote by $\pG_\SC$ 
the defining ideal of $\overline{\SC}$ in $\Zb_k/I_k$: it belongs to $\pspec(\Zb_k/I_k)$. 
Since $\ZC_k^\t$ has finitely many symplectic leaves, the map 
\equat\label{eq:poisson prime}
\fonctio{\symp(\ZC_k^\t)}{\pspec(\Zb_k/I_k)}{\SC}{\pG_\SC}
\endequat
is bijective (see~\eqref{eq:poisson prime general}). 
% The inverse is given as follows: 
% if $\pG \in \pspec(\Zb_k/I_k)$ corresponds to $\SC$ through this bijection, 
% then $\SC$ is the smooth locus of the closed irreducible subvariety defined by $\pG$. 

\bigskip

\subsection{${\boldsymbol{\t}}$-cuspidality} 
We define a {\it $\t$-cuspidal} symplectic 
leaf\footnote{This definition coincides with the notion of {\it cuspidal} 
leaf of $\ZC_k$ introduced by Bellamy~\cite[\S{5}]{bellamy cuspidal} in the case where $\t=1$.} 
to be a zero-dimensional symplectic leaf of $\ZC_k^\t$. 
We will therefore also call it a {\it $\t$-cuspidal} point. 
Through the bijection~\eqref{eq:poisson prime}, 
the set of $\t$-cuspidal points is naturally in bijection with the set 
$\pmax(\Zb_k/\sqrt{I_k})$ of 
maximal ideals of the algebra $\CM[\ZC_k^\t]=\Zb_k/\sqrt{I_k}$ which are also Poisson ideals (note that 
$\pmax(\Zb_k/I_k)= \pmax(\Zb_k/\sqrt{I_k})\subset \pspec(\Zb_k/\sqrt{I_k})$). 

\bigskip

\begin{rema}\label{rem:cusp-fixe}
It follows from Remark~\ref{rem:symp-stable} that $\t$-cuspidal points 
are fixed under the action of $\CM^\times$.\finl
\end{rema}

\bigskip

We denote by $\Cus_k^\t(V,W)$ the set of pairs $(P,p)$ where $P$ 
is a $\t$-split parabolic subgroup of $W$ and $p$ is a 
$\t$-cuspidal point of $\ZC_{k_P}(V_P,P)^\t$, where $k_P$ denotes 
the restriction of $k$ to the parabolic subgroup $P$. The group $W_\t$ 
acts on $\Cus_k^\t(V,W)$ and we denote by $\Cus_k^\t(V,W)/W_\t$ 
the set of its orbits in $\Cus_k^\t(V,W)$. 
If $(P,p) \in \Cus_k^\t(V,W)$, we denote by $[P,p]$ its $W_\t$-orbit. 

\bigskip

\subsection{Main result}
With the above notation, Theorem~A can be restated (and made more precise) as follows:

\bigskip

\noindent{\bf Theorem A.}
{\it There is a natural bijection (which will be explicitly constructed in Section~\ref{sec:tau-hc})
$$\fonctio{\Cus_k^\t(V,W)/W_\t}{\symp(\ZC_k^\t)}{[P,p]}{\SC_{P,p}.}$$
It satisfies that $\Upsilon_k(\SCov_{P,p})$ is the image of $(V^P)^\t \times (V^{*P})^\t$ 
in $V/W \times V^*/W$. In particular,
$$\dim \SC_{P,p} = 2 \dim (V^P)^\t.$$}

\bigskip

We will prove Theorem~A in the next sections. 
First, in Section~\ref{sec:0}, we will recall the proof, essentially 
due to Brown-Gordon~\cite[Prop.~7.4]{brown gordon} of Theorem~A 
whenever $k=0$ and $\t=1$. In Section~\ref{sec:0-tau}, we will use 
Lehrer-Springer Theorem~\ref{theo:lsp} to prove Theorem~A 
whenever $k=0$. In Section~\ref{sec:one-parameter}, we will use a deformation 
argument to attach to each symplectic leaf a $W_\t$-orbit of $\t$-split 
parabolic subgroups: in some sense, this is half of the construction 
of the above bijection. The second half will be constructed in 
Section~\ref{sec:tau-hc}, where the proof of Theorem~A 
will be completed.

Let us also restate Conjecture~B:

\bigskip

\centerline{\begin{minipage}{0.88\textwidth}
\noindent{\bf Conjecture B.} {\it Let $(P,p) \in \Cus_k^\t(V,W)$. 
Then there exists $l \in \paramset((V^P)^\t,\Nrmov_{W_\t}(P_\t))$ 
and a $\CM^\times$-equivariant isomorphism of Poisson varieties
$$\overline{\SC}_{P,p}^\nor \simeq \ZC_l((V^P)^\t,\Nrmov_{W_\t}(P_\t)).$$} 
\end{minipage}}

\bigskip

\section{Symplectic leaves of ${\boldsymbol{\ZC_0=(V \times V^*)/W}}$}\label{sec:0}

\medskip

The Poisson bracket on $\Zb_0=\CM[V \times V^*]^W$ is the one obtained by restriction from 
the usual Poisson bracket on $\CM[V \times V^*]$. The symplectic leaves of $\ZC_0$ 
have been described in~\cite[Prop.~7.4]{brown gordon}: 
we recall their description in this section, and give some more precision about the 
structure of their closure.

\def\Nrmo{{\overline{\Nrm}}}

\medskip

If $P \in \para(W)$, let $\VC\!\VC^*(P)$ denote the set of elements $(v,v^*) \in V \times V^*$ 
such that $W_v \cap W_{v^*} = P$. Again, the family $(\VC\!\VC^*(P))_{P \in \para(W)}$ 
is a stratification of $V \times V^*$ (the order between strata corresponds to 
the reverse order of the inclusion of parabolic subgroups). If $\PG \in \para(W)/W$, 
we let $\UC\!\UC^*(\PG)$ denote the image of $\VC\!\VC^*(P)$ in $(V \times V^*)/W$, 
where $P$ is any element of $\PG$. 
Then $(\UC\!\UC^*(\PG))_{\PG \in \para(W)/W}$ is a stratification of $(V \times V^*)/W$ 
(the order between strata corresponds to the reverse order of the inclusion, 
up to conjugacy, of parabolic subgroups).

Fix now $\PG \in \para(W)/W$ and $P \in \PG$. Then 
\equat\label{eq:vv*}
\VC(P) \times \VC^*(P) \subset \VC\!\VC^*(P) \subset V^P \times V^{*P}.
\endequat
Note that $\Nrmov_W(P)$ acts on $V^P \times V^{*P}$ and that $\VC\!\VC^*(P)$ is the open 
subset of $V^P \times V^{*P}$ on which it acts freely. The image of 
$\overline{\VC\!\VC^*(P)}=V^P \times V^{*P}$ is equal to $\overline{\UC\!\UC^*(\PG)}$. 

Recall from~\S\ref{sub:poisson-fixed} that $V^P \times V^{*P}$ is not a Poisson 
subvariety of $V \times V^*$ but that it inherits from $V \times V^*$ a Poisson 
structure. This Poisson  structure is the natural one endowed by the product of
a vector space with its dual: it is $\Nrmov_W(P)$-equivariant, so 
$(V^P \times V^{*P})/\Nrmov_W(P)$ is also a Poisson variety. By definition, 
$\Nrmov_W(P)$ acts freely on the open subset $\VC\!\VC^*(P)$, so 
the variety $\VC\!\VC^*(P)/\Nrmov_W(P)$ is smooth and its Poisson bracket 
makes it a symplectic variety. The next proposition is a particular case 
of the discussion preceding Proposition~\ref{prop:leaves-lisse}:

\bigskip

\begin{lem}\label{lem:quotient-z0}
Let $\PG \in \para(W)/W$ and let $P \in \PG$. Then:
\begin{itemize}
\itemth{a} The closed subvariety $\overline{\UC\!\UC^*(\PG)}$ is a Poisson subvariety 
of $(V \times V^*)/W$.

\itemth{b} The map $\VC\!\VC^*(P) \to \UC\!\UC^*(\PG)$ 
induces an isomorphism
$$\VC\!\VC^*(P)/\Nrmov_W(P) \longiso \UC\!\UC^*(\PG)$$
of Poisson varieties. 
\end{itemize}
\end{lem}

\bigskip

% \begin{proof}
% (a) Remark~\ref{rem:poisson variety} and Example~\ref{ex:sous-espace} imply that, 
% for any subgroup $P$ of $W$, the 
% natural restriction map $\CM[V \times V^*]^P \longto \CM[V^P \times V^{*P}]$ 
% is a morphism of graded Poisson algebras.
% The inclusion $\CM[V \times V^*]^W \subset \CM[V \times V^*]^P$ also respects the 
% Poisson bracket. So the kernel of the composition $\CM[V \times V^*]^W \to \CM[V^P \times V^{*P}]$ 
% is a Poisson and prime ideal of $\CM[V \times V^*]^W$: as it is the ideal of definition of 
% $\overline{\UC\!\UC^*(\PG)}$, this shows~(a).
% 
% \medskip
% 
% (b) The map $\VC\!\VC^*(P) \to \UC\!\UC^*(\PG)$ is surjective by definition. 
% It is injective since, if $z$ and $z'$ are two elements of $\VC\!\VC^*(P)$ such 
% that $z'=w(z)$ for some $w \in W$, then $w \in \Nrmov_W(P)$ because 
% $W_z=W_{z'}=P$. The fact that it is Poisson follows from the proof of~(a). 
% It remains to prove that it is an isomorphism. For this it is sufficient to show 
% that $\UC\!\UC^*(\PG)$ is smooth (as a bijective morphism between smooth 
% varieties is necessarily an isomorphism). But the morphism $(V \times V^*)/P \longto (V \times V^*)/W$ 
% is \'etale at points of $\VC\!\VC^*(P)/P$, so the property of being smooth can be read 
% in $(V \times V^*)/P$. But now the result follows from the fact that 
% $V^P \times V^{*P}$ is a smooth closed subvariety of $(V \times V^*)/P$.
% \end{proof}

\bigskip

\begin{coro}\label{coro:feuille-0-normalisation}
Let $\PG \in \para(W)/W$ and let $P \in \PG$. Then the above 
isomorphism $\VC\!\VC^*(P)/\Nrmov_W(P) \longiso \UC\!\UC^*(\PG)$ extends to an isomorphism 
of Poisson varieties
$$(V^P \times V^{*P})/\Nrmov_W(P) \longiso \overline{\UC\!\UC^*(\PG)}^\nor.$$
\end{coro}

\bigskip

\begin{proof}
The surjective map $\ph : (V^P \times V^{*P})/\Nrmov_W(P) \to \overline{\UC\!\UC^*(\PG)}$ 
induces an injection $\CM[\overline{\UC\!\UC^*(\PG)}]\subset \CM[V^P \times V^{*P}]^{\Nrmov_W(P)}$ 
between algebra of regular functions, and both algebras have the same fraction fields by 
Lemma~\ref{lem:quotient-z0}(b). But $\ph$ is finite and $\CM[V^P \times V^{*P}]^{\Nrmov_W(P)}$ 
is integrally closed, so $\CM[V^P \times V^{*P}]^{\Nrmov_W(P)}$ is the integral closure of 
$\CM[\overline{\UC\!\UC^*(\PG)}]$ in its fraction field. This completes the proof 
of the corollary.
\end{proof}

\bigskip

The next result follows immediately from Lemma~\ref{lem:quotient-z0} and is a particular 
case of Proposition~\ref{prop:leaves-lisse} (see 
also~\cite[Prop.~7.4]{brown gordon}):

\bigskip

\begin{prop}\label{prop:leaves-0}
The family $(\UC\!\UC^*(\PG))_{\PG \in \para(W)/W}$ of locally closed subvarieties 
is the stratification of $\ZC_0=(V \times V^*)/W$ by symplectic leaves.
\end{prop}

\bigskip

Let us interprete the results of this section in terms of Theorem~A 
and Conjecture~B for $k=0$ and $\t=\Id_V$. 
First, it follows from Corollary~\ref{coro:feuille-0-normalisation} 
that, if $\PG \in \para(W)/W$ and if $P \in \PG$, then 
$\dim \UC\!\UC^*(\PG)=2\dim V^P$. Therefore, $\UC\!\UC^*(\PG)$ is $\Id_V$-cuspidal 
(we will say {\it cuspidal} for simplification) if and only if $V^P=0$. 
Therefore, there is at most one cuspidal leaf of $\ZC_0$ and there is actually one 
if and only if $V^W=0$ (in this case, this cuspidal leaf will be simply denoted by $0$, 
as it is the $W$-orbit of $0 \in V \times V^*$). This shows that
$$\Cus_0^{\Id_V}(V,W)=\{(P,0)~|~P \in \para(W)\} \longbij \para(W).$$
Consequently, the bijection $\Cus_0^{\Id_V}(V,W)/W \longiso \symp(\ZC_0)$ predicted by 
Theorem~A in the case where $k=0$ and $\t=\Id_V$ is simply given by the formula
$$\SC_{P,0} = \UC\!\UC^*(\PG)$$
for all $\PG \in \para(W)/W$ and all $P \in \PG$: this is the content of 
Proposition~\ref{prop:leaves-0}. 
Moreover, Corollary~\ref{coro:feuille-0-normalisation} proves Conjecture~B 
in this case:

\bigskip

\begin{prop}\label{prop:0}
Theorem~A and Conjecture~B hold if $k=0$ and $\t=\Id_V$.
\end{prop}

\section{Symplectic leaves of ${\boldsymbol{\ZC_0^\t}}$}\label{sec:0-tau}

\medskip

We have $\ZC_0^\t=((V \times V^*)/W)^\t$. So, for studying its symplectic leaves, 
the next consequence of Lehrer-Springer Theorem~\ref{theo:lsp} will be crucial:

\bigskip

\begin{prop}\label{prop:springer-double}
The natural map 
$$\ii_\t : (V^\t \times V^{*\t})/W_\t \longto ((V \times V^*)/W)^\t=\ZC_0^\t$$
is a finite bijective morphism of Poisson varieties: it is the normalization 
of the variety $\ZC_0^\t$. 
\end{prop}

\bigskip

\begin{proof}
Only the statement on the bijectivity needs to be proved, the others 
being obvious or immediate consequences.
Let us first prove that $\ii_\t$ is injective. 
Let $(v_1,v_1^*)$ and $(v_2,v_2^*) \in V^{\t} \times V^{*\t}$ be 
such that $(v_2,v_2^*)$ belong to the $W$-orbit of $(v_1,v_1^*)$. 
Then there exists $a \in W$ such that $(v_2,v_2^*)=a(v_1,v_1^*)$. 
By Theorem~\ref{theo:lsp}(b), there exists $b \in W_{V^\t}^\setw$ 
such that $v_2=b(v_1)$. Therefore, $b^{-1}a(v_1)=v_1$ and 
$b^{-1}(v_2^*) = b^{-1}a(v_1^*)$. In other words, $b^{-1}a$ belongs 
to the stabilizer $W_{v_1}$ of $v_1$ in $W$ (it is a parabolic subgroup). 
Since $\t(v_1)=v_1$, $\t$ normalizes $W_{v_1}$. Hence, since 
$\t$ is $W_{v_1}$-full by~\eqref{eq:w-full}, we may apply Theorem~\ref{theo:lsp}(b) 
to the pair $(W_{v_1},\t)$ so that, by dualizing, 
there exists $c \in (W_{V^\t}^\setw)_{v_1}$ such that $b^{-1}(v_2^*)=c(v_1^*)$. 
Therefore, $bc \in W_{V^\t}^\setw$ and $bc(v_1,v_1^*)=(v_2,v_2^*)$, as desired.

\medskip

Let us now prove that $\ii_\t$ is surjective. Let $(v,v^*) \in V \times V^*$ 
be such that its $W$-orbit is $\t$-stable. By Theorem~\ref{theo:lsp}(b), 
there exists $x \in W$ such that $\t(x(v))=x(v)$. So, by replacing 
$(v,v^*)$ by $x(v,v^*)$ if necessary, we may, and we will, assume that 
$\t(v)=v$. Therefore, there exists $a \in W$ such that 
$(\t(v),\t(v^*))=(a(v),a(v^*))$. In other words, $a(v)=v$ and $\t(v^*)=a(v^*)$. 
So $a$ belongs to the parabolic subgroup $W_v$, which is $\t$-stable. 
Applying again Theorem~\ref{theo:lsp}(b) to $(W_v,\t)$ (since $\t$ is $W_v$-full 
by~\eqref{eq:w-full}), and dualizing, 
one gets that there exists $b \in W_v$ such that $\t(b(v^*))=b(v^*)$. 
Therefore, $ab(v,v^*) \in V^\t \times V^{*\t}$, as desired.
\end{proof}

\bigskip

\begin{rema}\label{rem:normal?}
We do not know if there are examples of pairs $(W,\t)$ 
such that the variety $\ZC_0^\t$ is not normal. By the above proposition, 
saying that $\ZC_0^\t$ is normal is equivalent to saying that any $W_\t$-invariant 
polynomial function on $V^\t \times V^{*\t}$ extends to a $W$-invariant 
polynomial function on $V \times V^*$.\finl
\end{rema}

\bigskip

% Whenever $W_{V^\t}^\ptw=1$, we follow Springer terminology and 
% say that $\t$ is {\it regular}. This is equivalent to say 
% that $V_\reg^\t\neq\vide$. In this case, the above Theorem~\ref{theo:lsp} 
% was first proved by Springer with the 
% following more precise version~\cite[Prop.~3.5~and~Theo.~4.2]{springer}:
% 
% \bigskip
% 
% \begin{theo}[Springer]\label{theo:springer}
% Assume that $\t$ is full and regular. Then $W_{V^\t}^\setw=W^\t$, and $W^\t$ acts as a reflection 
% group on $V^\t$ and the natural map $V^\t/W^\t \to (V/W)^\t$ 
% is an isomorphism of varieties.
% \end{theo}
% 
% \bigskip

\bigskip

A bijective morphism of Poisson varieties does not necessarily induce 
a bijection between symplectic leaves, but it turns out that this holds for our map $\ii_\t$, 
as shown by the Corollary~\ref{coro:uu-stratification} below. Before proving it, let us introduce some
notation. If $Q$ is a parabolic subgroup of $W_\t$, we denote by $\VC\!\VC^*_\t(Q)$ 
the set of pairs $(v,v^*) \in V^\t \times V^{*\t}$ such that $Q=W_v \cap W_{v^*}$. 
If $\QG$ denotes the $W_\t$-orbit of $Q$, we denote by $\UC\!\UC^*_\t(\QG)$ 
the image of $\VC\!\VC^*_\t(Q)$ in $(V^\t \times V^{*\t})/W_\t$. By Proposition~\ref{prop:leaves-0} 
applied to the pair $(V^\t,W_\t)$, the locally closed subvariety 
$\UC\!\UC^*_\t(\QG)$ is a symplectic leaf of 
$(V^\t \times V^{*\t})/W_\t$ and all the symplectic leaves are obtained in this way. 
Note first the following easy fact:

\bigskip

\begin{lem}\label{lem:double-tau}
Let $\PG \in \para(W)/W$, let $P \in \PG$ and let $w \in \Nrmov_W(P)$. Then:
\begin{itemize}
\itemth{a} $\VC\!\VC^*(P)^{w\t} \neq \vide$ if and only if $\VC(P)^{w\t} \neq \vide$.

\itemth{b} $\UC\!\UC^*(\PG)^\t \neq \vide$ if and only if $\UC(\PG)^\t \neq \vide$.
\end{itemize}
\end{lem}

\bigskip

\begin{proof}
Note that~(a) implies~(b) by Lemma~\ref{lem:quotient-z0}(b). On the other hand, 
if $\VC(P)^{w\t} \neq \vide$, then $\VC^*(P)^{w\t} \neq \vide$. So, if we pick 
$v \in \VC(P)^{w\t}$ and $v^* \in \VC^*(P)^{w\t}$, then $(v,v^*) \in \VC\!\VC^*(P)^{w\t}$. 
This proves the ``if'' part of~(a).

Conversely, if $\VC\!\VC^*(P)^{w\t} \neq \vide$, pick $(v,v^*) \in \VC\!\VC^*(P)^{w\t}$. 
Then $\VC^*(W_{v^*})^{w\t} \neq \vide$ so $\VC(W_{v^*})^{w\t}$. Pick $v' \in \VC(W_{v^*})^{w\t}$ 
and let $S$ denote the subspace of $V$ generated by $v$ and $v'$. Then $P=W_S^\ptw$, so there 
exists $v'' \in S$ such that $W_{v''} = P$. But $v'' \in V^{w\t} \cap \VC(P)$, 
which proves the ``only if'' part of~(a).
\end{proof}

\bigskip

The above Lemma allows to apply to the bijective morphism of varieties 
$\ii_\t : (V^\t \times V^{*\t})/W_\t \longto \ZC_0^\t$ the same arguments 
as in~\S\ref{sub:strat}. For instance, if $\PG \in \para(W)/W$, then 
it follows from Lemma~\ref{lem:double-tau} and Proposition~\ref{prop:stratification} 
that $\UC\!\UC^*(\PG)^\t \neq \vide$ if and only if $\PG$ contains a $\t$-split 
parabolic subgroup. 

Moreover, if $\PG \in (\para(W)/W)_\split^\t$ and if $P \in \PG$ is $\t$-split, 
then the $\t$-equivariant isomorphism $\UC\!\UC^*(\PG) \simeq \VC\!\VC^*(P)/\Nrmov_W(P)$ 
induces a decomposition into irreducible components
\equat\label{eq:uuptau}
\UC\!\UC^*(\PG)^\t =  \bigcup_{E \in \ECt_P\t/\Nrmov_W(P)} \VC\!\VC^*(P)^E/\Nrmov_W(P),
\endequat
where $\VC\!\VC^*(P)^E$ is defined in the same way as $\VC(P)^E$. Similarly, 
the analogue of Corollary~\ref{coro:uu-stratification} is given as follows:

\bigskip

\begin{prop}\label{prop:uu-stratification}
Let $\PG \in (\para(W)/W)_\split^\t$, let $P \in \PG$ and let $E \in \ECt_P\t/\Nrmov_W(P)$. 
Let $\PG_E$ denote the $W_\t$-orbit of $\t$-split parabolic subgroups of $W$ 
associated with $E$ through the bijection of Proposition~\ref{prop:param-tau-split}(b). 
Let $\QG_E$ denote the $W_\t$-orbit of parabolic subgroups of $W_\t$ of the form 
$Q_\t$ for $Q \in \PG_E$ (see Lemma~\ref{lem:p bijectif}). 
Then 
$$\ii_\t(\UC\!\UC^*_\t(\QG_E))=\VC\!\VC^*(P)^E/\Nrmov_W(P).$$
\end{prop}

\bigskip

\begin{coro}\label{coro:uu-stratification}
The bijective morphism of varieties $\ii_\t : (V^\t \times V^{*\t})/W_\t \longto \ZC_0^\t$ 
induces a bijection between symplectic leaves. 
\end{coro}

\bigskip

\begin{proof}
% Let $A$ denote the Poisson algebra $\CM[V^\t \times V^{*\t}]^{W_\t}$. 
Both varieties admits finitely many symplectic leaves so, by taking the closure, 
these leaves are, in both cases, in bijection with the set of irreducible closed Poisson 
subvarieties. 

Now, let $\SC$ be an irreducible closed Poisson subvariety of $(V^\t \times V^{*\t})/W_\t$. 
Since $\ii_\t$ respects the Poisson bracket, $\ii_\t(\SC)$ is also an irreducible closed 
Poisson subvariety of $\ZC_0^\t$: this shows that $\ii_\t$ induces an injective map 
between the symplectic leaves of $(V^\t \times V^{*\t})/W_\t$ and those of $\ZC_0^\t$. 

Let us now show that this map is surjective. For this, let $\SC$ be a symplectic leaf 
of $\ZC_0^\t$. Then there exists $\PG \in \para(W)/W$ such that 
$\SC \cap \UC\!\UC^*(\PG)^\t$ is open and dense in $\SC$. So $\PG$ contains a $\t$-split parabolic 
subgroup $P$ and the decomposition of $\UC\!\UC^*(\PG)^\t \neq \vide$ into 
irreducible components is given by~\eqref{eq:uuptau}. But $\UC\!\UC^*(\PG)$ is 
smooth and symplectic, so $\UC\!\UC^*(\PG)^\t$ is also smooth and symplectic, 
so $\SC \cap \UC\!\UC^*(\PG)^\t$ is equal to one of these irreducible 
components. The result then follows from Proposition~\ref{prop:uu-stratification}.
\end{proof}

\bigskip

\begin{prop}\label{prop:k0}
If $k=0$, then Theorem~A and Conjecture~B hold.
\end{prop}

\bigskip

\begin{proof}
By Corollary~\ref{coro:uu-stratification}, $\ZC_0^\t$ admits a $\t$-cuspidal point 
if and only if $(V^\t)^{W_\t}=0$ and, in this case, there is only one $\t$-cuspidal 
point, namely the orbit of $0$. So, still by Corollary~\ref{coro:uu-stratification} 
(and~\eqref{eq:vptau}), $\Cus_0^\t(V,W)$ is in bijection with conjugacy 
classes of parabolic subgroups of $W_\t$ and $\symp(\ZC_0^\t)$ is also 
in bijection with conjugacy classes of parabolic subgroups of $W_\t$. 
This provides a natural bijection between $\Cus_0^\t(V,W)$ and 
$\symp(\ZC_0^\t)$, which satisfies the required properties of Theorem~A.

\medskip

Let us now prove Conjecture~B in this case. So let $\SC$ be a symplectic 
leaf of $\ZC_0^\t$. Let $\LC=\ii_\t^{-1}(\SC)$: it is a symplectic 
leaf of $(V^\t \times V^{*\t})^{W_\t}$ and 
$$\MCov=\ii_\t^{-1}(\SCov).$$
Since $\ii_\t$ is bijective, we have $\LCov^\nor=\SCov^\nor$, 
so Conjecture~B now follows from Lemma~\ref{lem:quotient-z0} 
applied to the pair $(V^\t,W_\t)$ instead of $(V,W)$.
\end{proof}

\bigskip

\begin{rema}\label{param-0}
If $k=0$, then the parameter $l$ involved in Conjecture~B is equal to $0$.\finl
\end{rema}

\bigskip

\section{Parabolic subgroups attached to symplectic leaves}\label{sec:one-parameter}

\medskip

\subsection{D\'efinition}
Let $\SC$ be a symplectic leaf of $\ZC_k^\t$. We denote by $\pG_\SC$ the prime ideal 
of $\Zb_k$ defining $\SCov$: then $\pG_\SC^\t \in \pspec(\Zb_k^\t)$. Now, 
the isomorphism $\gr_\FC \Zb_k \simeq \Zb_0$ is $\t$-equivariant and 
$(\gr_\FC \Zb_k)^\t=\gr_\FC(\Zb_k^\t)$. So $\gr_\FC(\pG_\SC^\t)$ is an 
ideal of $\Zb_0^\t$. The next important result follows mainly from~\cite[Theo.~2.8]{martino}.

\bigskip

\begin{lem}\label{lem:grfp}
The ideal $\sqrt{\gr_\FC(\pG_\SC^\t)}$ of $\Zb_0^\t$ is prime, Poisson and contains $I_0^\t$.
\end{lem}

\bigskip

\begin{proof}
First, the Poisson bracket $\{,\}$ on $\Zb_k$ is a {\it proto-Poisson bracket} 
of degree $-2$ in the sense of~\cite[Def.~2.4]{martino} and its 
associated graded Poisson bracket on $\Zb_0$ is also the natural 
Poisson bracket on $\Zb_0$ (for a proof of both facts, see~\cite[Lem~.2.26]{EG}). 

The same facts also holds by taking fixed points under the $\t$-action 
and so, the fact that $\sqrt{\gr_\FC(\pG_\SC^\t)}$ is a prime ideal of $\Zb_0^\t$ 
which is Poisson is an application of~\cite[Theo.~2.8]{martino}. 

Finally, $\t$ acts trivially on $\Zb_k/\pG_\SC$, so it acts trivially 
on $\gr_\FC(\Zb_k/\pG_\SC)=\gr_\FC(\Zb_k)/\gr_\FC(\pG_\SC)=\Zb_0/\gr_\FC(\pG_\SC)$. 
This shows that $\gr_\FC(\pG_\SC)$ contains $I_0$ and so $\gr_\FC(\pG_\SC^\t)$ 
contains $I_0^\t$.
\end{proof}

\bigskip

Lemma~\ref{lem:grfp} shows that $\sqrt{\gr_\FC(\pG_\SC^\t)}$ defines a symplectic 
leaf $\SC_0$ of $\ZC_0^\t$ so, by Corollary~\ref{coro:uu-stratification}, there exists 
a unique $W_\t$-orbit $\QG_\SC$ of parabolic subgroups of $W_\t$ such 
that $\sqrt{\gr_\FC(\pG_\SC^\t)}$ is the defining ideal of $\ii_\t(\UC\!\UC_\t^*(\QG_\SC))$. 
Through the bijection of Lemma~\ref{lem:p bijectif}, there exists a unique 
$W_\t$-orbit $\PG_\SC$ of $\t$-split parabolic subgroups of $W$ such that 
$\QG_\SC=\{P_\t~|~P \in \PG_\SC\}$. 

\bigskip

\begin{defi}\label{defi:parabolique}
Let $\SC$ be a symplectic leaf of $\ZC_k^\t$. The $W_\t$-orbit of $\t$-split parabolic subgroups 
$\PG_\SC$ is called the $W_\t$-orbit {\bfit associated with $\SC$}. Any element of $\PG_\SC$ 
is called an {\bfit associated $\t$-split parabolic subgroup} (with $\SC$). 
\end{defi}

\bigskip

If $P$ is a $\t$-split parabolic subgroup of $W$ associated with $\SC$, then 
\equat\label{eq:dim feuille}
\dim \SC = 2 \dim(V^P)^\t.
\endequat
Indeed, $\dim \SC = \dim \SC_0$.

\bigskip

\subsection{Geometric construction} 
Let $\pi : V/W \times V^*/W \to V/W$ and $\pi^\ve : V/W \times V^*/W \longto V^*/W$ 
denote the first and second projection respectively. 
The next proposition gives another characterization of the $W_\t$-orbit of 
$\t$-split parabolic subgroups associated with a symplectic leaf:

\bigskip

\begin{prop}\label{prop:image upsilon}
Let $\SC$ be a symplectic leaf of $\ZC_0^\t$. Then:
\begin{itemize}
\itemth{a} $\Upsilon_k(\overline{\SC})=\iota_\t(\overline{\UC_\t(\QG_\SC)}) \times \iota_\t^\ve(\overline{\UC_\t^*(\QG_\SC)})$.

\itemth{b} $\pi(\Upsilon_k(\overline{\SC}))=\iota_\t(\overline{\UC_\t(\QG_\SC)})$.

\itemth{c} $\pi^\ve(\Upsilon_k(\overline{\SC}))=\iota_\t^\ve(\overline{\UC_\t^*(\QG_\SC)})$.
\end{itemize}
\end{prop}

\bigskip

\begin{proof}
Let $\Hb_k^\#$ denote the $\CM[\hbar]$-algebra obtained as the 
quotient of the algebra $\CM[\hbar] \otimes (\Trm(V \oplus V^*) \rtimes W)$ 
by the relations 
\equat\label{eq:rels-lambda}
\begin{cases}
[x,x']=[y,y']=0,\\
[y,x]=\hbar^2 \DS{\sum_{H\in\mathcal{A}} \sum_{j=0}^{e_H-1}
e_H(k_{H,i}-k_{H,i+1}) 
\frac{\langle y,\a_H \rangle \cdot \langle \a_H^\ve,x\rangle}{\langle \a_H^\ve,\a_H\rangle} \e_{H,i}},
\end{cases}
\endequat
for all $y$, $y' \in V$ and $x$, $x' \in V^*$. It follows from the comparison of 
the relations~\eqref{eq:rels} and~\eqref{eq:rels-lambda} that there is a well-defined 
morphism of $\CM[\hbar]$-algebras $\th : \Hb_k^\# \longto \rees_\FC \Hb_k$ such that
$$\th(y)=\hbar y,\qquad \th(x)=\hbar x\qquad \text{and}\qquad 
\th(w)=w$$
for all $y \in V$, $x \in V^*$ and $w \in W$. In fact,
\equat\label{eq:iso-rees}
\text{\it $\th$ is an isomorphism of algebras.}
\endequat
Indeed, the surjectivity is immediate while the injectivity follows from 
the PBW decomposition~\eqref{eq:pbw}, which also holds for $\Hb_k^\#$, 
namely, the map 
$$\CM[\hbar] \otimes \CM[V] \otimes \CM W  \otimes \CM[V^*] \longto \Hb_k^\#$$
induced by the multiplication is an isomorphism of $\CM[\hbar]$-modules. 
% The algebra $\Hb_k^\#$ comes equipped with the following structures:
% 
% \medskip
% 
% \begin{quotation}
% \begin{itemize}
% \itemth{A1} An action of $\CM^\times \times \CM^\times$ given by 
% $$
% \begin{cases}
% (\xi,\xi') \cdot x = \xi\xi' x & \text{if $x \in V^*$,}\\
% (\xi,\xi') \cdot y = \x^{-1}\xi' y & \text{if $y \in V$,}\\
% (\xi,\xi') \cdot w = w & \text{if $w \in W$,}\\
% (\xi,\xi') \cdot \hbar = \xi'\hbar.
% \end{cases}$$
% The action of the first copy of $\CM^\times$ extends the one which has been already 
% defined in~\S\ref{subsub:action}.
% 
% \medskip
% 
% \itemth{A2} An action of $\t$ (which acts trivially on $\hbar$ and naturally on 
% $\Trm(V \oplus V^*) \rtimes W$): it commutes with the action of $\CM^\times \times \CM^\times$. 
% 
% \medskip
% 
% \itemth{S} Specialization at $\l \in \CM$: we have a canonical isomorphism of algebras 
% $\CM[\hbar]/\langle \hbar -\l \rangle \simeq \Hb_{\l^2 k}$.
% \end{itemize}
% \end{quotation}
% 
% \medskip

Let $\Zb_k^\#$ denote the centre of $\Hb_k^\#$. Then it follows from~\eqref{eq:iso-rees} 
that $\th$ induces an isomorphism of algebras 
\equat\label{eq:iso-rees-z}
\Zb_k^\# \longto \rees_\FC \Zb_k.
\endequat
Again, $\Zb_k^\#$ is a flat family of deformations 
of $\Zb_0=\CM[V \times V^*]^W$. We denote by $\ZC_k^\#$ the affine variety such that 
$\CM[\ZC_k^\#]=\Zb_k^\#$. The inclusion $\Pb \injto \Zb_k^\#$ induces a morphism 
$\Upsilon_k^\# : \ZC_k^\# \longto \PC=V/W \times V^*/W$. 

The action of $\t$ and $\CM^\times$ extends easily to $\Hb_k^\#$, by letting 
them act trivially on the indeterminate $\hbar$. However, $\Hb_k^\#$ (and so $\Zb_k^\#$) 
inherits an extra-action of $\CM^\times$. Namely, 
there is an action of $\CM^\times \times \CM^\times$ given by 
$$
\begin{cases}
(\xi,\xi') \cdot x = \xi\xi' x & \text{if $x \in V^*$,}\\
(\xi,\xi') \cdot y = \x^{-1}\xi' y & \text{if $y \in V$,}\\
(\xi,\xi') \cdot w = w & \text{if $w \in W$,}\\
(\xi,\xi') \cdot \hbar = \xi'\hbar.
\end{cases}$$
The action of the first copy of $\CM^\times$ extends the one which has been already 
defined in~\S\ref{subsub:action}. Through the isomorphism $\th$, 
this action on $\rees_\FC \Hb_k$ is just the restriction 
of the action on $\CM[\hbar] \otimes \Hb_k$ given by
$$(\xi,\xi') \cdot (P(\hbar) \otimes h)=P(\xi'\hbar) \otimes (\xi \cdot h).$$
Then~\eqref{eq:rees} can be retrieved thanks to the isomorphism $\th$, the 
specialization process~(S) and~\eqref{eq:rees-facile}. 

Specializing $\hbar$ to $\l \in \CM$ gives the algebras $\Hb_{\l^2k}$ 
and $\Zb_{\l^2 k}$. Geometrically, the inclusion $\CM[\hbar] \injto \Zb_k^\#$ 
induces a flat morphism $\ZC_k^\# \to \CM$ whose fiber at $\l$ is the Calogero-Moser 
space $\ZC_{\l^2 k}$. 

Now, view $\SC$ as a subvariety of $\ZC_k^\#$ and 
let $\SCov_0=\overline{(1 \times \CM^\times)\cdot \SC} \cap \ZC_0$, 
endowed with its reduced structure. Then, using the isomorphism $\th$, 
it follows from the definition of the action of the second copy of $\CM^\times$ 
on $\rees_\FC(\Hb_k)$ that the defining ideal of $\overline{\SC}_0$ is 
$\sqrt{\gr_\FC(\pG_\SC)}$. 

Since $\SC$ is $(\CM^\times \times 1)$-stable, we have 
$$(1 \times \CM^\times)\cdot \SC=(\CM^\times \times \CM^\times)\cdot \SC = \D(\CM^\times) \cdot \SC,$$
where $\D(\CM^\times)$ is the diagonal in $\CM^\times \times \CM^\times$. 
Note also that, if $\xi \in \CM^\times$ 
and $z \in \ZC_k^\#$, then $\pi \circ \Upsilon_k^\#((\xi,\xi) \cdot z)=\pi \circ \Upsilon_k^\#(z)$. 
Moreover, if $z \in \overline{\SC}$, then $z_0=\lim_{\xi \to 0} (\xi,\xi) \cdot z$ exists 
(because the action $\D(\CM^\times)$ has non-negative weights) and $z_0 \in \overline{\SC}_0$. 
So $(\pi \circ \Upsilon_k)(z)=(\pi \circ \Upsilon_0)(z_0)$ belongs to 
$\iota_\t(\overline{\UC_\t(\QG_\SC)})$, as expected. This shows that
$$\pi(\Upsilon_k(\overline{\SC})) \subset \iota_\t(\overline{\UC_\t(\QG_\SC)}).$$
By exchanging the role of $V$ and $V^*$, we have 
$$\pi^\ve(\Upsilon_k(\overline{\SC})) \subset \iota_\t^\ve(\overline{\UC_\t^*(\QG_\SC)}).$$
Therefore, 
$$\Upsilon_k(\overline{\SC}) \subset \iota_\t(\overline{\UC_\t(\QG_\SC)}) 
\times \iota_\t^\ve(\overline{\UC_\t^*(\QG_\SC)}).$$
Since $\Upsilon_k(\SCov)$ is closed irreducible 
of dimension $2 \dim(V^P)^\t$ (by~\eqref{eq:dim feuille} and 
the finiteness of the morphism $\Upsilon_k$), we get that 
$$\Upsilon_k(\overline{\SC}) = \iota_\t(\overline{\UC_\t(\QG_\SC)}) 
\times \iota_\t^\ve(\overline{\UC_\t^*(\QG_\SC)}).$$
In other words, this proves~(a). Now, (b) and~(c) follow from~(a).
\end{proof}

\bigskip

Keep the notation introduced in the above proof ($\Hb_k^\#$, $\Zb_k^\#$, $\ZC_k^\#$,\dots) 
and let us explain how this proof provides a justification of Conjecture~B as well as a possible 
strategy for proving it. Indeed, if $\SC$ is a symplectic leaf of $\ZC_k$, let 
$\SCov^\#=\overline{(1 \times \CM^\times)\cdot \SC}$. Then $\SCov^\#$ comes equipped 
with a morphism $\varpi : \SCov^\# \longto \CM$ and we denote by $\nu : (\SCov^\#)^\nor \longto \CM$ 
the composition of the normalization morphism $(\SCov^\#)^\nor \longto \SCov^\#$ with $\varpi$. 
Then $\nu$ is flat~\cite[Chap.~III,~Prop.~9.7]{hartshorne}. 
Since $\varpi^{-1}(\CM^\times) \simeq \CM^\times \times \SCov$, we have
$$\nu^{-1}(\CM^\times) \simeq \CM^\times \times \SCov^\nor.$$
Let $\SCov_0^\bigstar$ denote the scheme-theoretic 
fiber of $\nu$ at $0$. Assume that we are able to show the following two facts:
\begin{itemize}
\itemth{1} The reduced subscheme of $\SCov_0^\bigstar$ is the normalization of $\SCov_0$.

\itemth{2} The scheme $\SCov_0^\bigstar$ is generically reduced.
\end{itemize}
Then a Theorem of Hironaka~\cite[Chap.~III,~Theo.~9.11]{hartshorne} 
would show that $\nu$ is a flat family of schemes, all of whose scheme-theoretic fibers are 
reduced, irreducible and normal varieties. As $\SCov_0^\bigstar=\SCov_0^\nor$ is the normalization 
of $\SCov_0$ by~(1), this would imply that $\SCov^\nor$ is a Poisson deformation 
of $\SCov_0^\nor$. So Conjecture~B would then follow from 
Propositions~\ref{prop:uu-stratification},~\ref{prop:k0} and a result of 
Bellamy~\cite[Theo.~1.4]{bellamy counting} 
(which follows works of Ginzburg-Kaledin~\cite{GK} and Namikawa~\cite{namikawa 1},~\cite{namikawa 2}).

% If $J$ is an ideal of $\Zb_k$, we denote by $J^\#$ the inverse image of $\rees_\FC(J)$ 
% in $\Zb_k^\#$ through the isomorphism~\eqref{eq:iso-rees-z} and we set 
% $J_0=(J^\# + \langle \hbar \rangle)/\langle \hbar \rangle$: it is an ideal of 
% $\Zb_0$. We have 
% \equat\label{eq:grad}
% J_0 = \grad_\FC(J)
% \endequat
% (through~\eqref{eq:rees-facile},~\eqref{eq:iso-rees-z} and~\eqref{eq:filtration-h}).
% Note the following useful results:
% 
% \bigskip
% 
% \begin{lem}\label{lem:crucial}
% Let $I$ and $J$ be two ideals of $\Zb_k$. Then:
% \begin{itemize}
% \itemth{a} The algebra $\Zb_k^\#/J^\#$ is a flat $\CM[\hbar]$-algebra. 
% 
% \itemth{b} If $\{\Zb_k,J\} \subset I$, then $\{\Zb_0,J_0\} \subset I_0$.
% \end{itemize}
% \end{lem}
% 
% \bigskip
% 
% \begin{proof}
% (a) Through the isomorphism $\th$, we have by definition 
% $$\Zb_k^\#/J^\# \simeq \rees_\FC(\Zb_k)/\rees_\FC(J).$$
% But $\rees_\FC(\Zb_k)/\rees_\FC(J) \simeq \rees_\FC(\Zb_k/J)$, for the 
% filtration on $\Zb_k/J$ induced by the one on $\Zb_k$ (i.e. 
% $\FC_j(\Zb_k/J)=(J+\FC_j \Zb_k)/J$. But it is well-known 
% that $\rees_\FC(M) \simeq \CM[\hbar] \otimes M$ (as $\CM[\hbar]$-modules) 
% for any filtered $\CM$-vector space.
% 
% \medskip
% 
% (b) Through the isomorphism $\th$, we have $\Zb_0/J_0=\grad_\FC(\Zb_k)/\grad_\FC(J)$. 
% 
% 
% \end{proof}

\bigskip

% The construction will be done in two steps. 
% The first one consist in attaching to a symplectic leaf $\SC$ 
% a $\t$-split parabolic subgroup $P$ (well-defined up to $W_\t$-conjugacy): 
% this will be done in~\S{sub:para} by understanding the image $\Upsilon(\SC)$ in 
% $(V/W)^\t \times (V^*/W)^\t$. No completion is required for this step. 
% The second step is to explain how to attach to $\SC$ a $\t$-cuspidal point 
% of $\ZC_k(V_P,P)^\t$: this steps requires the use of completions of rings. 
% It will then remain to prove that the construction gives a bijective map.

\section{$\taub$-Harish-Chandra theory of symplectic leaves}\label{sec:tau-hc}

\medskip

Let $P$ be a parabolic subgroup of $W$ and let $\PG$ denote its conjugacy class.  
Let $k_P$ denote the restriction of $k$ to the hyperplane arrangement of $P$ 
and let $k_P^\circ$ denote its ``extension by zero'' to the hyperplane 
arrangement of $\Nrm_W(P)$: in other words, if $H \in \AC(V,\Nrm_W(P))$ and 
$0 \le i \le e_H-1$, we set 
$$(k_P^\circ)_{H,i}=
\begin{cases}
k_{H,i} & \text{if $H \in \AC(V,P)$,}\\
0 & \text{otherwise.}
\end{cases}
$$
If $\XC$ is a locally closed subvariety of $V/W$, we denote by $\reallywidehat{\ZC_k(V,W)}_\XC$ 
the scheme equal to the completion of $\ZC_k(V,W)$ at it locally closed subvariety 
$(\pi \circ \Upsilon_k)^{-1}(\XC)$. Note that it inherits a Poisson structure 
from the one of $\ZC_k(V,W)$ (see for instance~\cite[Lem.~3.5]{bellamy cuspidal}). 

Our construction of the {\it $\taub$-Harish-Chandra theory of symplectic leaves} 
will follow from an upcoming result of Bellamy-Chalykh~\cite{be-ch} 
which says that there is a natural isomorphism of Poisson schemes\footnote{As~\cite{be-ch} is still 
not published, we mention here that it is based on {\it Bezrukavnikov-Etingof like} constructions 
of isomorphisms when completing at a single point of $V/W$.}
\equat\label{eq:bellamy}
\reallywidehat{\ZC_k(V,W)}_{\UC(\PG)} \simeq \reallywidehat{\ZC_{k_P^\circ}(V,\Nrm_W(P))}_{\UC(\PG)}.
\endequat
Note that $\UC(\PG) \simeq \VC(P)/\Nrm_W(P)$ maybe be viewed as a locally closed 
subvariety of both $V/W$ and $V/\Nrm_W(P)$. The construction of this isomorphism 
implies that it is $\t$-equivariant.

A sheafified version of Proposition~\ref{prop:completion-fixe} implies that 
we can take fixed points under the action of $\t$ in the above isomorphism 
and get an isomorphism
\equat\label{eq:bellamy plus}
\reallywidehat{\ZC_k(V,W)^\t}_{\UC(\PG)^\t} \simeq 
\reallywidehat{\ZC_{k_P^\circ}(V,\Nrm_W(P))^\t}_{\UC(\PG)^\t},
\endequat
with obvious notation. Moreover, this isomorphism is also Poisson: indeed, the Poisson structure 
on the left-hand side comes from the Poisson structure on the quotient 
scheme $(\reallywidehat{\ZC_k(V,W)}_{\UC(\PG)})/\langle \t \rangle$ 
and one can use Corollary~\ref{coro:completion-G} 
(and similary for the right-hand side).

Now, the irreducible (i.e. connected in this case) components of $\UC(\PG)^\t$ have been 
described in Corollary~\ref{coro:stratification}: this leads to a decomposition 
of the two schemes involved in isomorphism~\eqref{eq:bellamy plus}. We focus on the 
irreducible component $\iota_\t(\UC_\t(\PG_\split^\t))$ of $\UC(\PG)^\t$ 
and get an isomorphism of Poisson schemes
\equat\label{eq:bellamy plus plus}
\reallywidehat{\ZC_k(V,W)^\t}_{\iota_\t(\UC_\t(\PG_\split^\t))} \simeq 
\reallywidehat{\ZC_{k_P^\circ}(V,\Nrm_W(P))^\t}_{\iota_\t(\UC_\t(\PG_\split^\t))}.
\endequat
A sheafified version of~\cite[Lem.~3.3,~3.4,~3.5]{bellamy cuspidal} provides 
a natural bijection between Poisson reduced irreducible subschemes of $\ZC_k(V,W)^\t$ 
of dimension $2\dim(V^P)^\t$ meeting $(\pi \circ \Upsilon_k)^{-1}(\iota_\t(\UC_\t(\PG_\split^\t))$ 
and Poisson reduced irreducible subschemes of 
$\reallywidehat{\ZC_k(V,W)^\t}_{\iota_\t(\UC_\t(\PG_\split^\t))}$ of dimension $2\dim(V^P)^\t$. 
A similar bijection is obtained with the right-hand side of~\eqref{eq:bellamy plus plus}. 
Using the isomorphism~\eqref{eq:bellamy plus plus}, one gets a bijection between the following two sets:
\begin{itemize}
\item[$\bullet$] The set $\SC_{\PG_\split^\t}$ 
of symplectic leaves of $\ZC_k(V,W)^\t$ of dimension $2\dim(V^P)^\t$ 
meeting $(\pi \circ \Upsilon_k)^{-1}(\iota_\t(\UC_\t(\PG_\split^\t))$;

\item[$\bullet$] The set $\SC_{\PG_\split^\t}'$ 
of symplectic leaves of $\ZC_{k_P^\circ}(V,\Nrm_W(P))^\t$ of dimension $2\dim(V^P)^\t$ 
meeting $(\pi' \circ \Upsilon_k')^{-1}(\iota_\t(\UC_\t(\PG_\split^\t))$.
\end{itemize}
Here, the maps $\pi'$ and $\Upsilon_k'$ are the analogues of $\pi$ and $\Upsilon_k$ for the 
Calogero-Moser space $\ZC_{k_P^\circ}(V,\Nrm_W(P))$. 

But it follows from Proposition~\ref{prop:image upsilon} that $\SC_{\PG_\split^\t}$ 
is exactly the set of symplectic leaves $\SC$ of $\ZC_k(V,W)^\t$ such that 
$\PG_\SC=\PG_\split^\t$. So Theorem~A will follow from the next lemma:

\bigskip

\begin{lem}\label{lem:fin}
The set $\SC_{\PG_\split^\t}'$ is in natural bijection with the set of $\Nrm_{W_\t}(P_\t)$-orbits 
of cuspidal points of $\ZC_{k_P}(V_P,P)^\t$. 
\end{lem}

\bigskip

\begin{proof}
Since $k_P^\circ$ is the extension by zero of $k_P$, we have 
$$\ZC_{k_P^\circ}(V,\Nrm_W(P))=\ZC_{k_P}(V,P)/\Nrmov_W(P) 
= (V^P \times V^{*P} \times \ZC_{k_P}(V_P,P))/\Nrmov_W(P).$$
Consequently, 
$$\UC(\PG) \times_{V/\Nrm_W(P)} \ZC_{k_P^\circ}(V,\Nrm_W(P))=
(\VC(P) \times V^{*P} \times (0 \times_{V_P/P} \ZC_{k_P}(V_P,P)))/\Nrmov_W(P).$$
As in~\eqref{eq:uptau} and~\eqref{eq:uuptau}, the $\t$-fixed points of 
$(\VC(P) \times V^{*P} \times (0 \times_{V_P/P} \ZC_{k_P}(V_P,P)))/\Nrmov_W(P)$ 
decomposes into pieces indexed by $\ECt_P\t/\Nrmov_W(P)$ as follows:
\begin{multline*}
\Bigl((\VC(P) \times V^{*P} \times (0 \times_{V_P/P} \ZC_{k_P}(V_P,P)))/\Nrmov_W(P)\Bigr)^\t= \\
\bigcup_{w \in [\ECt_P\t/\Nrmov_W(P)]} (\VC(P)^{w\t} \times (V^{*P})^{w\t} 
\times \ZC_{k_P}(V_P,P)^{w\t})/\Nrmov_W(P)^{w\t}.
\end{multline*}
Here, $[\ECt_P\t/\Nrmov_W(P)]$ is a set of representatives of $\ECt_P\t/\Nrmov_W(P)$. 
We may, and we will, assume that $1 \in [\ECt_P\t/\Nrmov_W(P)]$. Then, by construction, 
only the piece indexed by $1$ meets $(\pi' \circ \Upsilon_k')^{-1}(\iota_\t(\UC_\t(\PG_\split^\t))$. 
Therefore, $\SC_{\PG_\split^\t}'$ is in natural bijection with the set of 
symplectic leaves of 
$$\XC=(\VC(P)^{\t} \times (V^{*P})^{\t} \times \ZC_{k_P}(V_P,P)^{\t})/\Nrmov_W(P)^{\t}$$
of dimension $2\dim(V^P)^\t$. 
But $\Nrmov_W(P)^{\t}=\Nrmov_{W_\t}(P_\t)$ by Lemma~\ref{lem:normalisateurs}, and it acts freely on 
$\VC(P)^{\t} \times (V^{*P})^{\t} \times \ZC_{k_P}(V_P,P)^{\t}$. 
So it follows from Corollary~\ref{coro:finitude-bis} 
that the set of symplectic leaves of $\XC$ is in natural 
bijection with the set of $\Nrmov_{W_\t}(P_\t)$-orbits of symplectic leaves of 
$$\YC=\VC(P)^{\t} \times (V^{*P})^{\t} \times \ZC_{k_P}(V_P,P)^{\t}.$$ 
But any symplectic leaf of $\YC$ is of the form $\VC(P)^{\t} \times (V^{*P})^{\t} \times \SC$, 
where $\SC$ is a symplectic leaf of $\ZC_{k_P}(V_P,P)^{\t}$. For dimension reasons, 
$\SC_{\PG_\split^\t}'$ is in natural bijection with the set of $\Nrmov_{W_\t}(P_\t)$-orbits of 
symplectic leaves of $\ZC_{k_P}(V_P,P)^{\t}$ if dimension $0$, which is exactly 
the desired statement.
\end{proof}

\section{Examples}

\medskip

\subsection{Smooth case}
Assume in this subsection, and only in this subsection, that 
$\ZC_k$ is smooth and that $\t$ is of the form $\z w$ for some 
root of unity $\z$ and $w \in W$. We denote by $d$ the order of $\z$. Then
\equat\label{eq:bete}
\ZC_k^\t = \ZC_k^{\mub_d}.
\endequat
Since $\ZC_k$ is smooth, it is symplectic by~\eqref{eq:poisson prime}, and so $\ZC_k^\t$ 
is also smooth and symplectic: its symplectic leaves are exactly its 
irreducible (i.e. connected) components. 

In~\cite{bonnafe maksimau}, Maksimau and the author have described the irreducible 
components of $\ZC_k^{\mub_d}$ as particular Calogero-Moser spaces, and the reader can check 
that this description is compatible with Conjecture~B: however, it is not 
proved that the isomorphism preserves the 
Poisson structure. In other words~\cite[Theo.~2.13~and~5.1]{bonnafe maksimau}:

\bigskip

\begin{theo}
If $\ZC_k$ is smooth and $\t \in \CM^\times W$, then 
Conjecture B holds, possibly up to the Poisson structure.
\end{theo}

\bigskip

\subsection{Type ${\boldsymbol{G_4}}$} 
Thiel and the author~\cite{bonnafe thiel} have 
developed algorithms for computing presentations of $\Zb_k$ that have been implemented 
in {\sc Magma}~\cite{magma} (more precisely, in the {\sc Champ} package for {\sc Magma} 
written by Thiel~\cite{thiel}). This allows computations for (very) small groups.

We assume in this subsection, and only in this subsection, that $W$ is the group $G_4$, 
in Shephard-Todd classification~\cite{ST}. Then a presentation of $\Zb_k$ 
can be obtained with {\sc Magma} 
(see for instance~\cite[\S{5}]{bonnafe maksimau} or~\cite[Theo.~5.2]{bonnafe thiel}) 
and it has been checked in~\cite[Theo~4.7]{bonnafe thiel} that Conjecture~B holds in this case:

\bigskip

\begin{theo}
If $W=G_4$, then Conjecture~B holds.
\end{theo}

\bigskip

\subsection{Type ${\boldsymbol{B}}$} 
Assume in this subsection, and only in this subsection, that $W=W_n$ is a Coxeter 
group of type $B_n$ for some $n \ge 2$ (i.e. we may assume that $W=G(2,1,n)$ 
in Shephard-Todd classification) and that $\t=\Id_V$. Let $t=\diag(-1,1,\dots,1) \in W_n$ 
and, for $1 \le j \le n-1$, let $s_j$ denote the permutation matrix corresponding to 
the transposition $(j,j+1)$. 

There are two conjugacy classes of reflections: the class of $t$ (which 
generates an elementary abelian normal subgroup of order $2^n$) and the one of $s_1$  
(which generates a normal subgroup $W_n'=G(2,2,n)$ of index $2$ isomorphic to a Coxeter group 
of type $D_n$). 
We set $b=c_k(t)$ and $a=c_k(s_1)$ and we denote by $I_n$ the set of $m \in \ZM$ 
such that $|m| \le n-1$. The Dynkin diagram, together with the values of the parameter 
function $c_k$, is given as follows:
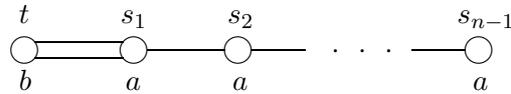
\begin{center}
\begin{picture}(250,40)
\put( 40, 20){\circle{10}}
\put( 44, 17){\line(1,0){33}}
\put( 44, 23){\line(1,0){33}}
\put( 81, 20){\circle{10}}
\put( 86, 20){\line(1,0){29}}
\put(120, 20){\circle{10}}
\put(125, 20){\line(1,0){20}}
\put(155, 17){$\cdot$}
\put(165, 17){$\cdot$}
\put(175, 17){$\cdot$}
\put(185, 20){\line(1,0){20}}
\put(210, 20){\circle{10}}
\put( 38, 30){$t$}
\put( 76, 30){$s_1$}
\put(116, 30){$s_2$}
\put(202, 30){$s_{n{-}1}$}
\put( 38, 5){$b$}
\put( 78, 5){$a$}
\put(118, 5){$a$}
\put(208, 5){$a$}
\end{picture}
\end{center}
The case where $a=0$ is somewhat uninteresting, as then $\ZC_k \simeq (C_b)^n/\SG_n$, 
where $C_b$ is the Calogero-Moser associated with the cyclic group of order $2$ 
whose equation is given by $C_b=\{(x,y,z) \in \CM^3~|~z^2=xy+4b^2\}$. 
So we assume throughout this subsection that $a \neq 0$.
This implies that 
\equat\label{eq:B-lisse}
\text{\it $\ZC_k$ is smooth if and only if $b/a \not\in I_n$.}
\endequat
As $\t=\Id_V$, the smooth case is uninteresting so we assume that $b/a=m \in I_n$. 
As the cases $b/a=m$ and $b/a=-m$ are equivalent, we also may assume that $m \ge 0$. 
The Calogero-Moser space $\ZC_k$ will then we denoted by $\ZC_{a,ma}(n)$. 
Symplectic leaves have been parametrized by Martino in his PhD Thesis~\cite[\S{5.4}]{martino these}. 
Bellamy and Thiel have then reinterpreted his result in terms of Bellamy parametrization 
{\it \`a la Harish-Chandra}~\cite[Lem.~6.5]{bellamy thiel}. This can be summarized as follows:
\begin{itemize}
\item[$\bullet$] $\ZC_{a,ma}(n)$ admits a cuspidal point if and only if there exists $r \in \ZM_{\ge 0}$ 
such that $n=r(r+m)$: if so, there is only one cuspidal point that we denote by $p_n$.

\item[$\bullet$] Therefore, $\Cus_{\Id_V}(\ZC_{a,ma}(n))=\{(W_{r(r+m)},p_{r(r+m)})~|~r(r+m) \le n\}$, 
with the convention that $W_0=1$ and $W_1=\langle t \rangle$.

\item[$\bullet$] If $r(r+m) \le n$, we denote by $\SC_r^m(n)$ the symplectic leaf of 
$\ZC_{a,ma}(n)$ associated with $(W_{r(r+m)},p_{r(r+m)})$ through the bijection of 
Theorem~A (since we are in the case where $\t=\Id_V$, this bijection 
was established by Bellamy~\cite{bellamy cuspidal} and Losev~\cite{losev}). We have 
$$\dim \SC_r^m(n)=2(n-r(r+m)).$$
\end{itemize}
If $r(r+m) \le n$, then 
$$\Nrmov_{W_n}(W_{r(r+m)}) \simeq W_{n-r(r+m)}.$$
Using the description of $\ZC_{a,ma}(n)$ in terms of quiver varieties, 
Bellamy, Maksimau and Schedler proved the following result~\cite{be-sc}:

\bigskip

\begin{theo}[Bellamy-Maksimau-Schedler]\label{theo:bellamy}
If $r(r+m) \le n$, then there is a $\CM^\times$-equivariant isomorphism of Poisson varieties 
$$\overline{\SC_r^m(n)} \simeq \ZC_{a,(m+2r)a}(n-r(r+m)).$$
\end{theo}

\bigskip

Note that Theorem~\ref{theo:bellamy} contains in particular the fact that 
$\overline{\SC_r^m(n)}$ is normal, so coincides with 
its normalization $\overline{\SC_r^m(n)}^\nor$.

\bigskip

\begin{coro}
Conjecture~B holds if $W$ is a Coxeter group of type $B_n$ and $\t=\Id_V$.
\end{coro}

\bigskip

\subsection{Type ${\boldsymbol{D}}$}
Assume in this subsection, and only in this subsection, that $W=W_n'$ is a Coxeter 
group of type $D_n$ for some $n \ge 4$ (i.e. we may assume that $W=G(2,2,n)$ 
in Shephard-Todd classification). We set $a=c_k(s_1)$, as in the 
previous subsection, and the Calogero-Moser space $\ZC_k$ will be denoted 
by $\ZC_a'(n)$. The case where $a=0$ being treated in Section~\ref{sec:0-tau}, 
we assume throughout this subsection that $a \neq 0$. 
The following facts are proved in~\cite[Theo.~7.2]{bellamy thiel}\footnote{Note 
that there is a little mistake in~\cite[Theo.~7.2]{bellamy thiel}, which can be easily 
corrected to give the statement written here.}:
\begin{itemize}
\item[$\bullet$] $\ZC_a'(n)$ admits a cuspidal point if and only if there exists 
$r \in \ZM_{\ge 0} \setminus \{1\}$ such that 
$n=r^2$: if so, there is only one cuspidal point that we denote by $p_n'$.

\item[$\bullet$] Therefore, 
$\Cus_{\Id_V}(\ZC_a'(n))=\{(W_{r^2}',p_{r^2}')~|~0 \le r^2 \le n~\text{and}~r \neq 1\}$, 
with the convention that $W_0'=1$.

\item[$\bullet$] If $0 \le r^2 \le n$ and $r \neq 1$, we denote by $\SC_r'(n)$ the symplectic leaf of 
$\ZC_a'(n)$ associated with $(W_{r^2}',p_{r^2}')$. We have 
$$\dim \SC_r'(n)=2(n-r^2).$$
\end{itemize}
Let us give another description, coming from the link between $\ZC_a'(n)$ and 
the Calogero-Moser space $\ZC_{a,ma}(n)$ of type $B_n$ defined in the previous section 
for the special value $m=0$. 
Indeed, $\ZC_a'(n)$ admits an action of the element 
$t \in W_n \subset \Nrm_{\Gb\Lb_\CM(V)}(W_n')$ defined 
in the previous subsection and~\cite[Prop.~4.17]{bellamy thiel}
$$\ZC_a'(n)/\langle t \rangle \simeq \ZC_{a,0}(n),$$
as Poisson varieties endowed with a $\CM^\times$-action. Denote by 
$\g_n : \ZC_a'(n) \to \ZC_{a,0}(n)$ the quotient morphism.

\bigskip

\begin{prop}\label{prop:bd}
We have
$$\SC_0'(n)=\g_n^{-1}(\SC_0^0(n) \cup \SC_1^0(n))\qquad\text{and}\qquad
\SC_r'(n)=\g_n^{-1}(\SC_r^0(n))$$
for all $r \ge 2$ such that $r^2 \le n$.
\end{prop}

\smallskip

\begin{proof}
The symplectic leaves of $\ZC_a'(n)$ are characterized by their dimension, 
so every symplectic leaf is $t$-stable (so is the inverse image, under $\g_n$, 
of its image in $\ZC_{a,0}(n)$). But if $0 \le r^2 \le n$ and $r \neq 1$, then 
$\g_n(\overline{\SC_r'(n)})$ is a closed irreducible Poisson 
subvariety of $\ZC_{a,0}(n)$, so it is the closure of a symplectic leaf. 
For dimension reason, it must me equal to $\overline{\SC_r^0(n)}$. The result follows.
\end{proof}

\medskip

\begin{coro}\label{coro:bd}
We have
$$\ZC_a'(n)^t=\g_n^{-1}(\overline{\SC_1^0(n)}).$$
In particular, if $4 \le r^2 \le n$, then $t$ acts trivially on $\SC_r'(n)$.
\end{coro}

\smallskip

\begin{proof}
First, $t$ does not act trivially on $\ZC_a'(n)$ so it does not act trivially 
on the open leaf $\SC_0'(n)$. Since $\SC_0'(n)$ is smooth 
and symplectic, the description of the symplectic leaves of $\SC_0'(n)/\langle t \rangle$ 
is given by Proposition~\ref{prop:leaves-lisse}. But 
$\SC_0'(n)/\langle t \rangle=\SC_0^0(n) \cup \SC_1^0(n)$ by Proposition~\ref{prop:bd}. 
Comparing both descriptions shows that $t$ acts freely on $\SC_0^0(n)$ and 
trivially on $\SC_1^0(n)$. 

Therefore, $t$ acts trivially on the closure of $\g_n^{-1}(\SC_1^0(n))$ and freely on 
$\SC_0^0(n)$. 
But the closure of $\SC_1^0(n)$ is the union of the $\SC_r^0(n)$ for $r \ge 1$ 
(see~\cite[Lem.~6.5]{bellamy thiel}). So the corollary follows now directly from 
Proposition~\ref{prop:bd}.
\end{proof}

\medskip

Assume now that $4 \le r^2 \le n$. Then Corollary~\ref{coro:bd} shows that 
$\SC_r'(n) \simeq \SC_r^0(n)$. Moreover, $\Nrmov_{W_n'}(W_{r^2}') \simeq W_{n-r^2}$. 
So Theorem~\ref{theo:bellamy} shows the following result:

\medskip

\begin{coro}\label{coro:D}
Conjecture~B holds if $W$ is a Coxeter group of type $D_n$ and $\t\in \{\Id_V,t\}$.
\end{coro}

\smallskip

\begin{proof}
For the case $\t=\Id_V$, the work has already been done. For the case where $\t=t$, 
one must notice that $\t$ is $W_n'$-regular so it is $W_n'$-full 
(see Example~\ref{ex:regular}), that $(W_n')_\t=(W_n')^\t \simeq W_{n-1}$ 
(see Example~\ref{ex:regular-w}) and that 
$$\Nrmov_{(W_n')_\t}((W_{r^2}')_\t) = \Nrmov_{W_{n-1}}(W_{r^2-1})\simeq W_{n-r^2}.$$
Then the result follows from Theorem~\ref{theo:bellamy} and Corollary~\ref{coro:bd}. 
\end{proof}

\medskip

\subsection{Dihedral groups at equal parameters}
Let $d$ be a natural number and let $\xi$ denote a primitive $2d$-th root of unity. 
For $j \in \ZM/2d\ZM$, we set 
$$s_j=\begin{pmatrix} 0 & \xi^j \\ \xi^{-j} & 0 \end{pmatrix}$$
We assume in this section, and only in this section, that $W= \langle s_0,s_2\rangle$ is dihedral 
of order $2d$ and that $\t=s_1$: note that $\t^2=\Id_V$, that $\t s_0 \t^{-1} = s_2$, 
that $\t s_2 \t^{-1}=s_0$ and that $\t$ is $W$-full. We set $a=c_k(s_0)$ and, since 
$k$ is $\t$-stable by hypothesis, we have $c_k(s_2)=a$. In other words, 
we are in the equal parameter case studied by the author in~\cite{bonnafe diedral 2}. 
Moreover, in~\cite[\S{4}]{bonnafe diedral 2}, the author determined the structure 
of $\ZC_k^\t$. This gives:

\bigskip

\begin{prop}
If $W$ is dihedral of order $2d$ and if $\t$ is as above, then Conjecture~B holds.
\end{prop}

\bigskip

% \newpage
% 
\def\sectionname{Appendix}\label{appendix}
\renewcommand\thesection{\Alph{section}}
\setcounter{section}{0}

\section{Completion and finite group actions}

\medskip

\boitegrise{{\bf Hypothesis and notation.} {\it 
We fix in this appendix a commutative noetherian $\CM$-algebra $R$, an ideal $I$ 
of $R$, a fintie group $G$ acting on the $\CM$-algebra $R$ and we assume that 
$I$ is $G$-stable. We set $J=\langle I^G \rangle_R$. 
\\
\hphantom{A} Let $\rG_G$ be the ideal of $R$ generated by 
the family $(r-g(r))_{r \in R, g \in G}$ and set $R(G)=R/\sqrt{\rG_G}$. 
We denote by $I(G)$ the image of $I$ in $R(G)$. Note that $R(G)$ is the biggest quotient 
algebra of $R$ which is reduced and on which $G$ acts trivially. \\
\hphantom{A} Finally, 
we denote by $\Rhat_I$ the $I$-adic completion of $R$, i.e.
$$\Rhat_I=\varprojlim_j R/I^j,$$
and by $\iota : R \to \Rhat_I$ the canonical map.}}{0.75\textwidth}

\medskip

The results of this Appendix do not pretend to any originality, 
and might certainly be written in greater generality. We nevertheless 
do not find appropriate references containing all of them, and decided 
to state them in terms which are suitable for our purpose.

\bigskip

\begin{lem}\label{lem:app-fixe}
There exists an integer $m$ such that $I^m \subset J$.
\end{lem}

\bigskip

\begin{proof}
Let $\pG$ be a prime ideal of $R$ containing $I^G$. We first wish to prove tha 
$\pG$ contains $I$. For this, let $r \in I$. Then $\prod_{g \in G} g(r) \in I^G$, 
and so there exists $g_r \in G$ such that $g_r(r) \in \pG$ because $\pG$ is prime. 
This shows that $I \subset \cup_{g \in G} g(\pG)$. By the Prime Avoidance Lemma, 
we get that there exists $g \in G$ such that $I \subset g(\pG)$. 
Since moreover $I$ is $G$-stable, we get that $I \subset \pG$. In other 
words, $I$ is contained in any prime ideal containing $J$. So $I \subset \sqrt{J}$. 
As $R$ is noetherian, the result follows from Levitsky's Theorem~\cite[Theo.~10.30]{lam}.
\end{proof}

\bigskip

\begin{lem}\label{lem:J-fixed}
Let $j \ge 0$. Then $(J^j)^G=(I^G)^j$.
\end{lem}

\bigskip

\begin{proof}
The inclusion $(I^G)^j \subset (J^j)^G$ is obvious. Conversely, 
let $r \in (J^j)^G$. Then there exists a finite set $E$, a family 
$(r_e)_{e \in E}$ of elements of $R$ and a family $(i_e^{(1)},\dots,i_e^{(j)})$ 
of $j$-uples of elements of $I^G$ such that
$$r=\sum_{e \in E} r_e i_e^{(1)}\cdots i_e^{(j)}.$$
Since $r$ is $G$-invariant, we have $r=(1/|G|)\sum_{g \in G} g(r)$, so 
$$r=\sum_{e \in E} \Bigl(\frac{1}{|G|}\sum_{g \in G} g(r_e)\Bigr) i_e^{(1)}\cdots i_e^{(j)}.$$
Hence, $r \in (I^G)^j$.
\end{proof}

\bigskip

Since $I$ and $J$ are $G$-stable, the completions $\Rhat_I$ and $\Rhat_J$ inherit 
a $G$-action.

\bigskip

\begin{coro}\label{coro:completion-G}
The $\CM$-algebras $(\Rhat_I)^G$ and $\widehat{(R^G)}_{I^G}$ are canonically 
isomorphic.
\end{coro}

\begin{proof}
As $I^m \subset J \subset I$ for some $m$ by Lemma~\ref{lem:app-fixe}, 
the completions $\Rhat_I$ and $\Rhat_J$ 
are canonically isomorphic, and the isomorphism is $G$-equivariant. This gives 
an isomorphism $(\Rhat_I)^G \simeq (\Rhat_J)^G$. So the result follows directly 
from Lemma~\ref{lem:J-fixed}, because $(R/J^j)^G=R^G/(J^j)^G$ since we 
work in characteristic zero.
\end{proof}

\bigskip

\begin{prop}\label{prop:completion-fixe}
Assume that $R$ is Nagata. Then 
$$\widehat{R(G)}_I=\Rhat_I(G).$$
\end{prop}

\bigskip
\def\rGh{\hat{\rG}}

\begin{proof}
Let $\rGh_G$ denote the completion of $\rG_G$ at $I$. Since $R$ is noetherian, 
$\rGh_G$ is the ideal of $\Rhat_I$ generated by $\rG_G$ and 
$$\widehat{(R/\rG_G)}_I = \Rhat_I/\rGh_G$$
(see for instance~\cite[\S{4}]{greco}). 
This shows that $G$ acts trivially on $\Rhat_I/\rGh_G$ and so 
$\rGh_G$ is the ideal of $\Rhat_I$ generated by $(g(r)-r)_{r \in \Rhat_I, g \in G}$. 

Moreover, as $R$ is Nagata, we have that $\sqrt{\rGh_G}=\widehat{\sqrt{\rG_G}}$ 
by~\cite[Coro.~14.8]{greco}. The proposition follows.
\end{proof}

\bigskip

\begin{exemple}\label{ex:nagata}
Assume that $R$ is a localization or a completion of a finitely generated algebra. 
Then $R$ is Nagata.\finl
\end{exemple}

\end{document}